\documentclass[12pt,reqno]{amsart}
\usepackage{color}
\usepackage{todonotes,comment,mathrsfs}
\usepackage{stackengine}
\usepackage{latexsym}
\usepackage{comment}
\usepackage{amsmath,amssymb,dsfont,amsfonts}
\usepackage{mathrsfs}
\usepackage{verbatim}
\usepackage{graphicx}
\usepackage{graphics}
\usepackage{geometry}
\usepackage{booktabs}
\usepackage[shortlabels]{enumitem}
\usepackage{xcolor}
\definecolor{olive}{rgb}{0.3, 0.4, .1}
\definecolor{fore}{RGB}{249,242,215}
\definecolor{back}{RGB}{51,51,51}
\definecolor{title}{RGB}{255,0,90}
\definecolor{dgreen}{rgb}{0.,0.6,0.}
\definecolor{gold}{rgb}{1.,0.84,0.}
\definecolor{JungleGreen}{cmyk}{0.99,0,0.52,0}
\definecolor{BlueGreen}{cmyk}{0.85,0,0.33,0}
\definecolor{RawSienna}{cmyk}{0,0.72,1,0.45}
\definecolor{Magenta}{cmyk}{0,1,0,0}

\newtheorem{proposition}{Proposition}[section]
\newtheorem{theorem}{Theorem}[section]

\newtheorem{corollary}{Corollary}[section]
\newtheorem{lemma}{Lemma}[section]
\newtheorem{remark}{Remark}[section]
\newtheorem{example}{Example}[section]

\numberwithin{equation}{section}

\allowdisplaybreaks

\title[]{The $L^2$-norm of the forward stochastic integral w.r.t. Fractional Brownian motion $H > \frac{1}{2}$}

\author{Alberto OHASHI$^1$}
\author{Francesco RUSSO$^2$}

\address{$1$ Departamento de Matem\'atica, Universidade de Bras\'ilia, 13560-970, Bras\'ilia - Distrito Federal, Brazil}\email{amfohashi@gmail.com}
\address{$2$ ENSTA Paris, Institut Polytechnique de Paris,
 Unit\'e de Math\'ematiques appliqu\'ees, 828, boulevard des Mar\'echaux, F-91120 Palaiseau, France}
 \email{francesco.russo@ensta-paris.fr}

\begin{document}

\begin{abstract}
In this article, we present the exact expression of the $L^2$-norm of the forward stochastic integral driven by the multi-dimensional fractional Brownian motion with parameter $\frac{1}{2} < H < 1$. The class of integrands only requires rather weak integrability conditions compatible w.r.t. a random finite measure whose density is expressed as a second-order polynomial of the underlying driving Gaussian noise. A simple consequence of our results is the exact expression of the $L^2$-norm for the pathwise Young integral. 
\end{abstract}

\maketitle



\section{Introduction}

Let $B = (B^{(1)}, \ldots, B^{(d)})$ be a $d$-dimensional Fractional Brownian motion (henceforth abbreviated by FBM) with exponent $\frac{1}{2}< H < \frac{1}{2}$ on a probability space $(\Omega,\mathcal{F},\mathbb{P})$. That is, $B$ is a $d$-dimensional Gaussian process with covariance

$$\mathbb{E}\big[ B^{(i)}_s B^{(j)}_t \big]= \frac{1}{2}\{t^{2H}+ s^{2H}-|t-s|^{2H}\}\delta_{ij},$$
for $1\le i,j\le d$ and $0\le s,t < \infty$. We set $R(s,t):= \frac{1}{2}\{t^{2H}+ s^{2H}-|t-s|^{2H}\}$ for $(s,t) \in \mathbb{R}_+^2$. In the sequel, increments of paths $t\mapsto f_t$ will be denoted by $f_{s,t}:= f_t-f_s$.

For a given positive constant $T$, let $g:[0,T]\times \mathbb{R}^d\rightarrow \mathbb{R}^d$ be a Borel function, where $\mathbb{R}^d$ is equipped with the usual inner product $\langle \cdot, \cdot \rangle$. For a given $\epsilon>0$, let us denote

\begin{equation}\label{Iminus}
I^-(\epsilon,g(B),dB):=\frac{1}{\epsilon}\int_0^T \langle g(s,B_s), B_{s,s+\epsilon}\rangle ds.
\end{equation}
In this paper, we study the existence of the forward stochastic integral in the sense of stochastic calculus of regularization (see e.g. \cite{russo1993forward} and \cite{Russo_Vallois_Book} and other references therein):
\begin{equation}\label{epsilonconv}
\int_0^T g(t,B_t)d^-B_t:=\lim_{\epsilon\downarrow 0} I^-(\epsilon,g,dB)
\end{equation}
in $L^2(\mathbb{P})$. More importantly, we present the exact expression of the $L^2(\mathbb{P})$-norm of (\ref{epsilonconv}) in terms of an explicitly second order-polynomial of the underlying FBM.
In fact the techniques we develop in this article are not limited to FBM and can be used to treat other Gaussian noises admitting a covariance measure structure in the sense of \cite{kruk2007}. Those processes are finite quadratic variation processes in
the sense of \cite{russo2000stochastic}, where relations between
the quadratic variation and the covariance structure was first
investigated.
For the sake of simplicity, in this paper, we will not discuss the technical aspects of this point further.



There are essentially three different approaches to construct stochastic integrals w.r.t.~FBM with $\frac{1}{2} < H < 1$: pathwise, Malliavin calculus and, more recently, stochastic sewing lemma. Since $B$ has $\gamma$-H\"older continuous paths for $\gamma < H$, then one can make use of a pathwise approach to define a stochastic integral in the sense of Young \cite{young}. Indeed, if a stochastic processes $\{u_t,t\ge 0\}$ whose trajectories are $\lambda$-H\"older continuous with $\lambda >1-H$, then the Riemann-Stieltjes integral $\int_0^T u_sdB_s$ exists for each trajectory (see \cite{young}). Another pathwise approach was proposed by M. Z\"ahle in \cite{zahle} who introduced a generalized Stieltjes integral using the techniques of fractional calculus. The Malliavin calculus approach has been extensively studied by many authors over the last years: the divergence operator associated with the Gaussian space of FBM provides an isometry between the Reproducing Kernel Hilbert space (henceforth abbreviated by RKHS) associated with FBM and the first chaos, and this is precisely the so-called Paley-Wiener integral for the class of deterministic functions in the RKHS. Recall that in the classical Brownian motion case $H=\frac{1}{2}$, the divergence operator coincides with the classical It\^o integral when the integrand is adapted and this motivated many authors to develop a Malliavin calculus approach for a stochastic calculus w.r.t.~FBM driving noise. In this direction, we refer the reader to e.g \cite{viensSK}, \cite{decreusefond}, \cite{alos2001stochastic}, \cite{kruk2007}, \cite{carmona2003} and the monograph \cite{biaginibook} for a very complete list of works in this topic. 

One standard strategy for the obtention of $L^2(\mathbb{P})$-bounds for stochastic forward integrals of the type (\ref{epsilonconv}) and also the similar object computed in terms of Riemann sums is the identification with the divergence operator plus a trace term associated with a Gaussian driving noise $B$. The key point is the multiplication rule between smooth random variables and the increments $B_{s,s+\epsilon}$ (see e.g. Prop. 1.3.3 in \cite{nualart2006}), 

\begin{eqnarray}
\label{st1}Y_s B_{s,s+\epsilon} &=& Y_s \int_0^\infty \mathds{1}_{[s,s+\epsilon]}(r) \boldsymbol{\delta} B_r\\
\label{st2}&=& \int_0^\infty Y_s \mathds{1}_{[s,s+\epsilon]}(r) \boldsymbol{\delta} B_r + \big \langle \mathbf{D}Y_s, \mathds{1}_{[s,s+\epsilon]}  \big \rangle_{\mathcal{H}},
\end{eqnarray}
where $\mathbf{D}$ denotes the Malliavin derivative operator, $\boldsymbol{\delta}$ is the divergence operator and $\langle \cdot, \cdot\rangle_{\mathcal{H}}$ is the inner product of the underlying RKHS $\mathcal{H}$. We refer the reader to e.g. \cite{nualart2006} for the definitions of the basic objects $\mathbf{D}$, $\boldsymbol{\delta}$ and $\mathcal{H}$ above. The passage from (\ref{st1}) to (\ref{st2}) requires very strong regularity conditions from the integrands which prevents this framework of solving SDEs. The connection between the forward stochastic integral in the sense of stochastic calculus via regularization \cite{russo1993forward}, pathwise Young and divergence operator was established in \cite{alos2003} and \cite{russo2007elements}. In this direction, see also \cite{kruk2007}, Chapter 5 in \cite{biaginibook} and \cite{nualart2006} and the monograph \cite{Russo_Vallois_Book}.       

More recently,
Matsuda and Perkowski \cite{matsuda}  have obtained some progress beyond the standard Young pathwise interpretation by using the Mandelbrot-Van Ness representation \cite{mand} in terms of a two-sided Brownian motion. Based on an extension of the stochastic sewing lemma \cite{Le}, \cite{matsuda} has proved the existence of the limit


\begin{equation}\label{Riemann}
\int_0^T g(B_s)dB_s = \lim_{|\Pi|\rightarrow 0}\sum_{t_i \in \Pi} g(B_{t_i})B_{t_i,t_{i+1}}
\end{equation} 
 in $L^p(\mathbb{P})$ for every $p \ge 1$,
 as the mesh $|\Pi|$ of an arbitrary partition $\Pi$ of $[0,T]$ vanishes, where $g$ is bounded.
 The contributions \cite{yaskov} and \cite{torres} proved previously
 existence of (\ref{Riemann}) when $g$ is a bounded variation function. 






This article aims to produce the exact $L^2(\mathbb{P})$-norm of the forward integral (\ref{epsilonconv}) under rather weak conditions of integrand processes of the form $Y_\cdot = g(\cdot, B_\cdot)$, where $g$ is a time-space Borel function satisfying only integrability conditions w.r.t. a random field of the form

\begin{equation}\label{randomfield}
\frac{\partial^2 R}{\partial t \partial s}(s,t)I_{d\times d} + \mathcal{W}(s,t),
\end{equation}  
where $\mathcal{W}(s,t)$ is a second-order polynomial of the underlying Gaussian noise. For a given Borel function $g:[0,T]\times \mathbb{R}^d\rightarrow \mathbb{R}^d$, $(\epsilon,\delta) \in (0,1)^2$ and $Y_\cdot  = g(\cdot, B_\cdot)$, we shall write

\begin{equation}\label{fplit}
\Big \langle I^-(\epsilon,Y,dB), I^-(\delta,Y,dB) \Big\rangle_{L^2(\mathbb{P})} =\mathbb{E}\int_{[0,T]^2 }  \big\langle Y_{s}\otimes Y_{t} , \Lambda^-(\epsilon,\delta;s,t) \big\rangle dsdt
\end{equation}
where $\otimes$ denotes the tensor product in $\mathbb{R}^d$ and 
$ \Lambda^-(\epsilon,\delta; \cdot) = (\Lambda^-(\epsilon,\delta; \cdot))$ with

\begin{equation}\label{Lambdaop}
\Lambda^{-,ij}(\epsilon,\delta;s,t):=\frac{1}{\epsilon\delta}\mathbb{E}\Big[B^{(i)}_{s,s+\epsilon}B^{(j)}_{t,t+\delta}\big| B_s,B_t\Big],
\end{equation}
for $(s,t) \in [0,T]^2$ and $1\le i,j\le d$. The main novel idea presented in this article is to explore the regularity inherited from the
projection operator
$\Lambda^-(\epsilon,\delta; \cdot)$ which allows us to calculate explicitly the $L^2(\mathbb{P})$-norm of the forward stochastic integral under rather weak integrability conditions w.r.t. the random field (\ref{randomfield}). The investigation of the double limit $(\epsilon,\delta)\downarrow 0$ (simultaneously to zero) of (\ref{fplit}) and (\ref{Lambdaop}) provides the existence of (\ref{epsilonconv}) and (\ref{randomfield}), respectively, and as a by-product it gives the exact $L^2(\mathbb{P})$-norm of the forward stochastic integral.



Theorem \ref{mainTH2} presents the main result of this article: the explicit description of the $L^2(\mathbb{P})$-norm of (\ref{epsilonconv}) under rather weak regularity conditions on integrands of the form $Y_\cdot = g(\cdot, B_\cdot)$. The fundamental object is the random field $\mathcal{W}$ described in (\ref{Wproc}) and constructed in Theorem \ref{mainTH1}. The process $\mathcal{W}$ is a second-order polynomial of the underlying FBM and it is a $q$-integrable random field for $1\le q < \min \Big \{ \frac{1}{2H-1}; \frac{1}{2-2H}\Big\}$. Based on the $q$-integrability of $\mathcal{W}$, our condition on $Y_\cdot= g(\cdot, B_\cdot)$ in Theorem \ref{mainTH2}

\begin{equation}\label{IntC}
(s,t)\mapsto Y_s\otimes Y_t \in L^{p}(\Omega\times [0,T]^2; \mathbb{R}^{d\times d})
\end{equation}
depends on a integrability parameter $p \in \big(\max\big\{\frac{1}{2-2H}; \frac{1}{2H-1} \big\}, + \infty\big]$ which is the conjugate exponent of $q$.
Clearly $Y$ is allowed  to be  unbounded,
  see in particular Example \ref{Youngre},
but also not locally bounded, therefore discontinuous.

By means of a standard completion argument, we stress that the exact $L^2(\mathbb{P})$-norm expression of the forward stochastic integral given in Theorem \ref{mainTH2} allows us to extend (\ref{epsilonconv}) to a larger class of integrands as elements of a natural Hilbert space of processes equipped with the inner-product inherited from the right-hand side of (\ref{isoFORMULA}). We refer the reader to \cite{OhashiRusso2} for this procedure in the singular case $\frac{1}{4} < H < \frac{1}{2}$. In the present regular case $\frac{1}{2} < H < 1$ of this article, the resulting Hilbert space will be a space which contains not only integrable functions but also distributions in time and space. We leave the complete analysis of this extension and the associated stochastic calculus arising from Theorem \ref{mainTH2} to a future project. A step in this direction was done in
\cite{hairerLi}.

Finally, we mention that we have chosen to consider only state-dependent integrands, though we believe an extension of Theorem \ref{mainTH2} to more general integrands should be possible.

\

\noindent \textbf{Notation}: In this article, $0 < T < \infty$ is a positive finite terminal time,

$$[0,T]^2_\star  := \{(s,t) \in [0,T]^2; s\neq t, s>0, t>0\},$$
$$\Delta_T:=\{(s,t) \in [0,T]^2_\star; 0< s < t\le T\},$$
and the increment of a one-parameter function $f$ is denoted by $f_{s,t}:= f_t - f_s$ for $(s,t) \in \mathbb{R}^2_+$ Throughout this article, we write $(\epsilon,\delta)\downarrow 0$ to denote $\epsilon \downarrow 0$ and $\delta \downarrow 0$ \textit{simultaneously}. We further write $A\lesssim B$ for two positive quantities to express an estimate of the form $A \le C B$, where $C$ is a generic constant which may differ from line to line. To emphasize the dependence of $C$ on some parameters $a$, $b$, . . ., we write $A\lesssim_{a,b} B$. 

The sign function is denoted by $\text{sgn}$. We also write $a\wedge b = \min\{a;b\}$ and $a\vee b = \max\{a;b\}$ for any real numbers $a,b$. Moreover, for any two natural numbers $(i,j)$, we write $\delta_{ij} = 1$ when $i=j$ and $\delta_{ij}=0$, otherwise. The identity matrix of size $d$ is denoted by $I_{d\times d}$ and $\top$ denotes the transpose of a matrix.

If $a = (a^i)^d_{i=1}, b = (b^j)_{j=1}^d \in \mathbb{R}^d$, then
the matrix $(a^ib^j)_{1\le i,j\le d}$ is the  tensor product 
$a\otimes b$.
$\langle \cdot, \cdot \rangle$ denotes the standard Frobenius inner product on the space of $r\times q$-matrices for $r,q\ge 1$. Whenever clear from the context, $|\cdot|$ denotes a norm of a finite-dimension vector space. For a given column vector $a$, $a^\top$ denotes the transpose of $a$. For a random variable $X$, we denote $\|X\|^p_p:= \mathbb{E}|X|^p$ for $1\le p < \infty$. Let $L^p(\Omega\times [0,T]^2; \mathbb{R}^n)$ be the $L^p$ space of $\mathbb{R}^n$-valued two-parameters processes w.r.t to the measure $d \mathbb{P}\times dsdt$, $n\ge 1$ and $1\le p \le \infty$. With a slight abuse of notation, we write $\| Z\|^p_p = \mathbb{E}\int_{[0,T]^2} |Z(s,t)|^pdsdt$ whenever clear from the context.

If $\theta = (\theta_i)_{i=1}^d$ and $\xi = (\xi_j)_{j=1}^d$ are two column random vectors, then the joint covariance matrix of $(\theta,\xi)$ is denoted by

$$\text{Cov} \big( \theta, \xi \big): = \text{cov}\big(\theta_i, \xi_j \big); 1\le i\le k,~1\le j\le \ell,$$
where $\text{cov}$ is the covariance operation. We will denote

$$\varphi(s,t):= R(s,t) - v(s),$$
for $(s,t) \in \Delta_T$, where $v(s) := R(s,s)$.

\section{Main results}
This section presents and discusses the main results of this paper. In the sequel, $\Theta_{s,t}$ is the determinant of the covariance matrix of the Gaussian vector $(B^{(1)}_s,B^{(1)}_t)$. The first main result of this paper is the following. 

\begin{theorem}\label{mainTH1}
If $\frac{1}{2} < H< 1$, then 
\begin{eqnarray}
\label{limIN}\Lambda(s,t)&:=&\lim_{(\epsilon,\delta)\downarrow 0}\frac{1}{\epsilon\delta}\mathbb{E}\Big[ B_{s,s+\epsilon}\otimes  B_{t,t+\delta}| B_{s},B_{t}\Big]
\\
\nonumber&=& \mathcal{W}(s,t) + \frac{\partial^2 R}{\partial t \partial s}(s,t)I_{d\times d }; \quad s\neq t. 
\end{eqnarray}
The limit in (\ref{limIN}) is considered in $L^q(\Omega\times [0,T]^2; \mathbb{R}^{d\times d})$ for $1\le q < \min \Big \{ \frac{1}{2H-1}; \frac{1}{2-2H}\Big\}$, where $\mathcal{W} = (\mathcal{W}^{ij})$ is given by 

\begin{eqnarray}
\label{Wproc}\mathcal{W}^{ij}(s,t)&:=& \lambda_{11}(s,t)\lambda_{21}(s,t)\big(B^{(i)}_sB^{(j)}_s-\text{cov}(B^{(i)}_s; B^{(j)}_s)\big) \\
\nonumber&+& \lambda_{11}(s,t)\lambda_{22}(s,t) \big( B^{(i)}_s B^{(j)}_t - \text{cov}(B^{(i)}_s; B^{(j)}_t) \big)\\
\nonumber&+& \lambda_{12}(s,t) \lambda_{21}(s,t) \big( B^{(i)}_t B^{(j)}_s - \text{cov}(B^{(i)}_t; B^{(j)}_s) \big)\\
\nonumber&+& \lambda_{12}(s,t)\lambda_{22}(s,t)\big(B^{(i)}_t B^{(j)}_t - \text{cov}(B^{(i)}_t; B^{(j)}_t) \big),
\end{eqnarray}
for $1\le i,j\le d$, where
$$\lambda_{11}(s,t) := \frac{1}{\Theta_{s,t}} \Big\{\frac{1}{2}v'(s) v(t) - R(s,t)\frac{\partial R}{\partial s}
 (s,t) \Big\},$$
$$\lambda_{12}(s,t):= \frac{1}{\Theta_{s,t}} \Big\{\frac{\partial R}{\partial s}(s,t)
 v(s) - R(s,t)\frac{1}{2}v'(s) \Big\},$$
$$\lambda_{21}(s,t):= \frac{1}{\Theta_{s,t}} \Big\{\frac{\partial R}{\partial t}
(t,s) v(t) - R(s,t)\frac{1}{2}v'(t) \Big\},$$
$$\lambda_{22}(s,t):=\frac{1}{\Theta_{s,t}} \Big\{\frac{1}{2}v'(t) v(s) - R(s,t)\frac{\partial R}{\partial t}
 (t,s) \Big\},$$
for $(s,t) \in [0,T]^2_\star$.   

\end{theorem}

The two-parameter process $\mathcal{W}$ on $[0,T]^2_\star$ might be interpreted as a zero-mean product of (non-Markovian) Nelson-type derivatives \cite{nelson} of the form  

\begin{equation}
\mathbf{D}^{\mathcal{A}_{s,t}} B_s \otimes \mathbf{D}^{\mathcal{A}_{s,t}} B_t - \mathbb{E}\Big[ \mathbf{D}^{\mathcal{A}_{s,t}} B_s \otimes \mathbf{D}^{\mathcal{A}_{s,t}} B_t\Big];\quad (s,t) \in [0,T]^2_\star.
\end{equation}
Here $\mathcal{A}_{s,t}$ is the sigma-algebra generated by the pair $(B_s,B_t)$ and $\mathbf{D}^{\mathcal{A}_{s,t}} $ is the \textit{stochastic derivative} w.r.t. $\mathcal{A}_{s,t}$ in the sense of Darses and Nourdin \cite{darses} and Darses, Nourdin and Peccati \cite{darses1} defined by the almost sure limit
$$ \mathbf{D}^{\mathcal{A}_{s,t}} B_s:= \lim_{\epsilon\downarrow 0} \frac{1}{\epsilon} \mathbb{E}\big[ B_{s,s+\epsilon} |B_s,B_t\big],$$
for every $(s,t) \in [0,T]^2_\star$. 

Theorem \ref{mainTH1} plays an important role in the obtention of exact $L^2(\mathbb{P})$-norm of the stochastic forward integrals as demonstrated by the following result. Let $\mathcal{H}$ be the RKHS equipped with the inner product 


\begin{equation}\label{rkhsnorm}
\langle f,g \rangle _{\mathcal{H}} := \int_{[0,T]^2} \langle f_s, g_t\rangle \frac{\partial^2 R}{\partial t \partial s}(s,t)dsdt,
\end{equation}
for $f,g:[0,T]\rightarrow \mathbb{R}^d \in \mathcal{H}$. See e.g. \cite{kruk2007} and \cite{pipiras} for further details on the space $\mathcal{H}$. 

In the sequel, in order to shorten notation, for $Y_\cdot = g(\cdot,B_\cdot)$, we denote 
$$Y^{\otimes_2}(s,t):= Y_s\otimes Y_t,$$ 
for $(s,t) \in [0,T]^2$.

Below, we recall that $I^-(\epsilon,Y,dB)$ is the forward approximation given by (\ref{Iminus}).  
\begin{theorem}\label{mainTH2}


Let $g:[0,T]\times \mathbb{R}^d\rightarrow \mathbb{R}^d$ be a Borel function and $Y_\cdot = g(\cdot, B_\cdot)$. Assume that $Y^{\otimes_2}\in L^p(\Omega\times [0,T]^2;\mathbb{R}^{d\times d})$ for $\max\big\{\frac{1}{2-2H}; \frac{1}{2H-1} \big\}<  p  \le \infty$. Then, $\lim_{\epsilon\downarrow 0}I^{-}(\epsilon,Y,B)$ exists in $L^2(\mathbb{P})$ and

\begin{eqnarray}\label{isoFORMULA}
\Big\| \int_0^T Y_td^{-}B_t\Big\|^2_2 &=& \mathbb{E}\| Y\|^2_{\mathcal{H}} + \mathbb{E}\int_{[0,T]^2} \big\langle Y^{\otimes_2}(s,t), \mathcal{W}(s,t)\big\rangle dsdt\\
\nonumber&=& \mathbb{E}\int_{[0,T]^2} \big\langle Y^{\otimes_2}(s,t), \Lambda(s,t)\big\rangle dsdt,
\end{eqnarray}
where $\Lambda$ has been defined in \eqref{limIN}.

\end{theorem}
\begin{corollary}\label{cormainTH2}


Let $g:[0,T]\times \mathbb{R}^d\rightarrow \mathbb{R}^d$ be a Borel function and $Y_\cdot = g(\cdot, B_\cdot)$. Assume that $Y^{\otimes_2} \in L^p(\Omega\times [0,T]^2;\mathbb{R}^{d\times d})$ for $\max\big\{\frac{1}{2-2H}; \frac{1}{2H-1} \big\}<  p  \le \infty$ and let $1\le q < \min \Big \{ \frac{1}{2-2H}; \frac{1}{2H-1}\Big\}$ be the conjugate exponent of $p$. Then, 
\begin{equation}\label{upperbound}
\Big\| \int_u^v Y_rd^{-}B_r\Big\|^2_2 \lesssim_{H,q,d} \|Y^{\otimes_2 }\|_p|v-u|^{2(H-1) + \frac{2}{q}},
\end{equation}
for every $(u,v) \in [0,T]^2$. 
\end{corollary}

As we have mentioned earlier, Theorem \ref{mainTH2} and Corollary \ref{cormainTH2} can be applied to a large class of locally unbounded
(therefore discontinuous)
processes. Concerning the case of the classical pathwise Young integral
we have the following.

\begin{example}\label{Youngre}
Fix $\frac{1}{2} < H < 1$ and $\max\big\{\frac{1}{2-2H}; \frac{1}{2H-1} \big\}<  p  < \infty$. Let $g:[0,T]\times \mathbb{R}^d\rightarrow \mathbb{R}^d$ be a Borel function and $Y_\cdot=g(\cdot, B_\cdot)$ satisfies the following:
\begin{itemize}
  \item $t\mapsto Y_t$ is $\delta$-H\"older continuous a.s. for $\delta + H >1$

  \item $\sup_{0\le s < t\le T}\frac{|Y_{s,t}|}{|t-s|^\delta} \in L^{2p}(\mathbb{P})$.
\end{itemize}
Then, the forward stochastic integral $\int_0^T Y_sd^-B_s$ is a pathwise Young integral (see e.g. Proposition 3 in \cite{russo2007elements}).

In this case the $L^2(\mathbb{P})$-norm of the Young integral is given by the right-hand side of (\ref{isoFORMULA}).
Indeed
the classical Young's regularity condition immediately implies
\begin{equation}\label{growthYoung}
|Y_s\otimes Y_s|\lesssim s^\delta t^\delta + t^\delta + s^\delta,
\end{equation}
for some $\delta>0$ such that $\delta +  H>1$. Property (\ref{growthYoung}) implies (\ref{IntC}) for $Y_\cdot = g(\cdot, B_\cdot)$ and therefore
$Y$ fulfills the assumption of
Theorem \ref{mainTH2}.


\end{example}




At this point, we explain some differences w.r.t. the singular case treated in \cite{OhashiRusso2}. In the singular case $\frac{1}{4} < H < \frac{1}{2}$, the convergence of the Stratonovich version $\Lambda^0(\epsilon,\delta; s,t)$ of (\ref{Lambdaop}) takes place only pointwise for each $(s,t) \in [0,T]^2_\star$ a.s., while Theorem \ref{mainTH1} shows that we do have functional $L^q$ convergence of $\Lambda^-(\epsilon,\delta; \cdot)$ to the process $\Lambda$ if $H > \frac{1}{2}$. Moreover, in contrast to the singular case, the random measure $\Lambda$ in Theorem \ref{mainTH1} is finite over the rectangle $[0,T]^2$. This allows us to get rid of additional regularity conditions on the increment of the integrand processes reminiscent from the Young's pathwise approach. We stress this phenomenon is related to the presence (or lack of) of the covariance measure structure studied in the articles \cite{kruk2007} and \cite{krukrusso} for the regular and the singular case, respectively.

\section{Preliminaries}
In this section we present some basic objects related to (\ref{Lambdaop}). The projection operator $\Lambda^-(\epsilon,\delta; \cdot)$ onto the sigma-algebra generated by $B_s,B_t$ can be completely characterized by Gaussian linear regression which we now describe in detail. Let us denote

$$\mathbb{B}^{(i)}_{s,t}: = \left(
                                                                                        \begin{array}{c}
                                                                                          B^{(i)}_s \\
                                                                                          B^{(i)}_{t} \\
                                                                                        \end{array}
                                                                                      \right)
\quad \mathbf{B}_{s,t} = \left(
                                 \begin{array}{c}
                                   \mathbb{B}^{(1)}_{s,t} \\
                                   \vdots \\
                                   \mathbb{B}^{(d)}_{s,t} \\
                                 \end{array}
                               \right),
$$
for $(s,t) \in [0,T]^2$ and $1\le i\le d$. Observe

$$
\Lambda^-(\epsilon,\delta; s,t)= \frac{1}{ \epsilon\delta}\mathbb{E}\big[ B_{s,s+\epsilon}\otimes B_{t,t+\delta}| \mathbf{B}_{s,t}\big]
$$
for $(s,t) \in [0,T]^2$.


For each $1\le i,j\le d$, we observe $\big( \mathbf{B}^\top_{s,t}, B^{(i)}_{s,s+\epsilon}, B^{(j)}_{t,t+\delta}\big)$ is a $2d+2$-dimensional Gaussian vector. Therefore, the classical linear regression analysis yields the  representation

\begin{equation}\label{pr}
\left(
  \begin{array}{c}
    B^{(i)}_{s,s+\epsilon} \\
    B^{(j)}_{t,t+\delta} \\
  \end{array}
\right)
 = \mathbb{E}\Bigg [ \left(
  \begin{array}{c}
    B^{(i)}_{s,s+\epsilon} \\
    B^{(j)}_{t,t+\delta} \\
  \end{array}
\right) \Big| \mathbf{B}_{s,t} \Bigg] + N_{s,t}(\epsilon,\delta),
\end{equation}
where $N_{s,t}(\epsilon,\delta)$ is a (zero-mean) 2-dimensional Gaussian vector independent from the sigma-algebra generated by $\mathbf{B}_{s,t}$. Moreover, we have the following representation

\begin{equation}\label{2dcexp1}
\mathbb{E}\Bigg [ \left(
  \begin{array}{c}
    B^{(i)}_{s,s+\epsilon} \\
    B^{(j)}_{t,t+\delta} \\
  \end{array}
\right) \Big| \mathbf{B}_{s,t} \Bigg] = \mathcal{N}^{\epsilon,\delta}_{s,t}(i,j) \Sigma^{-1}_{s,t} \mathbf{B}_{s,t},
\end{equation}
where

$$\mathcal{N}^{\epsilon,\delta}_{s,t}(i,j):=\text{Cov}\Big((B^{(i)}_{s,s+\epsilon}, B^{(j)}_{t,t+\delta})^\top; \mathbf{B}_{s,t}\Big), $$

$$\Sigma_{s,t}:=\text{Cov}\Big(\mathbf{B}_{s,t}; \mathbf{B}_{s,t}\Big).
$$
We can represent $\mathcal{N}^{\epsilon,\delta}_{s,t}(i,j)$ and $\Sigma_{s,t}^{-1}$ as follows:
$\Sigma_{s,t}$ is a $2d\times 2d$-square matrix partitioned into a  block diagonal form

$$\Sigma_{s,t} = \left(
                   \begin{array}{cccc}
                     \Sigma_{s,t}(1,1) & 0 & \ldots & 0 \\
                     0 & \Sigma_{s,t}(2,2) & \ldots & 0 \\
                     \vdots & \vdots & \ddots & \vdots \\
                     0 & 0 & \ldots & \Sigma_{s,t}(d,d) \\
                   \end{array}
                 \right),
$$
where

$$\Sigma_{s,t}(i,i):= \text{Cov} \big( \mathbb{B}^{(i)}_{s,t}; \mathbb{B}^{(i)}_{s,t} \big),$$
for $(s,t) \in [0,T]^2$ and $1\le i\le d$. By the very definition, for $(s,t) \in [0,T]^2_\star$, we have
$$\Sigma_{s,t}(i,i) = \left(
                   \begin{array}{cc}
                     v(s) & R(s,t) \\
                     R(s,t) & v(t) \\
                   \end{array}
                 \right),
$$

$$\text{det} (\Sigma_{s,t}(i,i)) = v(s) v(t) - R^2(s,t), $$


$$\Sigma^{-1}_{s,t}(i,i) = \frac{1}{\Theta_{s,t}} \left(
                   \begin{array}{cc}
                     v(t) & -R(s,t) \\
                     -R(s,t) & v(s) \\
                   \end{array}
                 \right).
$$

Here, in order to keep notation simple, we denote

$$\Theta_{s,t}:= \text{det}(\Sigma_{s,t}(i,i)),$$
for $(s,t) \in [0,T]^2_\star $. By construction,

$$\Sigma^{-1}_{s,t} = \left(
                   \begin{array}{cccc}
                     \Sigma^{-1}_{s,t}(1,1) & 0 & \ldots & 0 \\
                     0 & \Sigma^{-1}_{s,t}(2,2) & \ldots & 0 \\
                     \vdots & \vdots & \ddots & \vdots \\
                     0 & 0 & \ldots & \Sigma^{-1}_{s,t}(d,d) \\
                   \end{array}
                 \right).
$$
By the independence between $B^{(i)}$ and $B^{(j)}$ for $i\neq j$, we observe we can represent $\mathcal{N}^{\epsilon,\delta}_{s,t}(i,j)$ as 


$$
\mathcal{N}^{\epsilon,\delta}_{s,t}(i,j)=\left(
  \begin{array}{ccccc}
    \alpha^{i}_1(\epsilon,s,t) & \ldots  & \alpha^{i}_i(\epsilon,s,t) & \ldots  & \alpha^{i}_d(\epsilon,s,t)  \\
    \beta^{j}_1(\delta,s,t) & \ldots & \ldots \beta^{j}_j(\delta,s,t)  & \ldots & \beta^{j}_d(\delta,s,t) \\
  \end{array}
\right).
$$
Here, the first row is represented by $\alpha^{i}_\ell(\epsilon,s,t):=(0,0)$ for every $\ell\neq i$ and
$$\alpha^{i}_i(\epsilon,s,t):= \Big(\text{cov}\big(B^{(1)}_{s, s+\epsilon}; B^{(1)}_s\big), \text{cov}\big(B^{(1)}_{s, s+\epsilon}; B^{(1)}_{t}\big)\Big).$$
The second row is represented by $ \beta^{j}_m(\delta,s,t) := (0,0)$ for every $m\neq j$ and

$$\beta^{j}_j(\delta,s,t):= \Big(\text{cov}\big(B^{(1)}_{t, t+\delta}; B^{(1)}_s\big), \text{cov}\big(B^{(1)}_{t, t+\delta}; B^{(1)}_{t}\big)\Big).$$
From $(\alpha^{i}_i(s,t),\beta^j_j(s,t)$, we can construct the  submatrix of $\mathcal{N}^{\epsilon,\delta}_{s,t}(i,j)$, given by

$$
\mathcal{N}^{\epsilon,\delta}_{s,t}:=\left(
    \begin{array}{cc}
      \mathbf{n}_{11}(\epsilon,s) & \mathbf{n}_{12}(\epsilon,s,t) \\
      \mathbf{n}_{21}(\delta,s,t) &\mathbf{n}_{22}(\delta,s,t) \\
    \end{array}
  \right),
$$
where

$$\mathbf{n}_{11}(\epsilon,s):= R(s,s+\epsilon) - R(s,s),
\quad \mathbf{n}_{12}(\epsilon,s,t):=R(s+\epsilon,t) - R(s,t),
$$
and
$$\mathbf{n}_{21}(\delta,s,t): = R(s,t+\delta) - R(s,t), \quad \mathbf{n}_{22}(\delta,s,t):=R(t+\delta,t) - R(t,t),
$$
for $(\epsilon,\delta) \in (0,1)^2$ and $(s,t) \in [0,T]^2$.

Let us denote

\begin{equation}\label{2dcexp2}
\left(
    \begin{array}{c}
      Z^{1;ij}_{s,t}(\epsilon) \\
      Z^{2;ij}_{s,t}(\delta) \\
    \end{array}
  \right)
:= \mathbb{E}\Bigg [ \left(
  \begin{array}{c}
    B^{(i)}_{s,s+\epsilon} \\
    B^{(j)}_{t,t+\delta} \\
  \end{array}
\right) \Big| \mathbf{B}_{s,t} \Bigg],
\end{equation}
for $(s,t) \in [0,T]^2_\star$. From (\ref{2dcexp1}), the coordinates of the conditional expectation (\ref{2dcexp2}) are given by

\begin{equation}\label{Z1epsilon}
Z^{1;ij}_{s,t}(\epsilon) = \lambda_{11}(\epsilon,s,t)B^{(i)}_s + \lambda_{12}(\epsilon,s,t)B^{(i)}_{t}
 \end{equation}
and

\begin{equation}\label{Z2delta}
Z^{2;ij}_{s,t}(\delta)= \lambda_{21}(\delta,s,t)B^{(j)}_s
+ \lambda_{22}(\delta,s,t)B^{(j)}_{t},
\end{equation}
where we set

$$\lambda_{11}(\epsilon,s,t):=\frac{1}{\Theta_{s,t}}\Big\{\mathbf{n}_{11}(\epsilon,s)v(t) - \mathbf{n}_{12}(\epsilon,s,t)R(s,t)\Big\},$$
$$\lambda_{12}(\epsilon,s,t):=\frac{1}{\Theta_{s,t}} \Big\{\mathbf{n}_{12}(\epsilon,s,t) v(s) - \mathbf{n}_{11}(\epsilon,s)R(s,t)\Big\},$$
$$\lambda_{21}(\delta,s,t):=\frac{1}{\Theta_{s,t}} \Big\{\mathbf{n}_{21}(\delta,s,t) v(t) - \mathbf{n}_{22}(\delta,s,t)R(s,t)\Big\},$$
$$\lambda_{22}(\delta,s,t):=\frac{1}{\Theta_{s,t}} \Big\{\mathbf{n}_{22}(\delta,s,t) v(s)  - \mathbf{n}_{21}(\delta,s,t)R(s,t)\Big\},$$
for $(s,t) \in [0,T]^2_\star$.
We shall represent $\Lambda^-(\epsilon,\delta; \cdot)$ as follows.

\begin{lemma}\label{Sprerepr}
For any $(\epsilon,\delta) \in (0,1)^2$,
\begin{eqnarray*}
\Lambda^{-,{ij}}(\epsilon,\delta; s,t)&=& \frac{1}{\epsilon \delta}\Big[ Z^{1;ij}_{s,t}(\epsilon) Z^{2;ij}_{s,t}(\delta) - \mathbb{E}[
Z^{1;ij}_{s,t}(\epsilon) Z^{2;ij}_{s,t}(\delta)] \Big]\\
&+& \frac{1}{\epsilon \delta}\mathbb{E}\big[B^{(i)}_{s,s+\epsilon} B^{(j)}_{t,t+\delta}\big],
\end{eqnarray*}
for $(s,t) \in [0,T]^2_\star$ and $1\le i,j\le d$.
\end{lemma}
\begin{proof}
The proof follows the same lines of the proof of Lemma 3.1 in \cite{OhashiRusso2} and hence we omit the details. 
\end{proof}

\begin{remark}\label{pointwiseR}
Since $R$ vanishes on the axis $R(x,0) = R(0,y)=0$ for every $x,y$ and $\frac{\partial^2R}{\partial t\partial s}$ is continuous on $\mathbb{R}^2_+$ outside the diagonal, then

$$R(s,t) = \int_0^t \int_0^s \frac{\partial^2R}{\partial t\partial s}(a,b)dadb.$$
Therefore,

$$\frac{1}{\epsilon\delta} \mathbb{E}[B^{(i)}_{s,s+\epsilon}B^{(j)}_{t,t+\delta}] = \frac{1}{\epsilon \delta} \int_t^{t+\delta} \int_s^{s+\epsilon} \frac{\partial^2R}{\partial t\partial s}(a,b)dadb\delta_{ij}\rightarrow \frac{\partial^2R}{\partial t\partial s}(s,t)\delta_{ij},$$
as $(\epsilon,\delta)\downarrow 0$ for each $(s,t) \in [0,T]^2_\star$ and $1\le i,j\le d$.

\end{remark}

\begin{lemma}\label{Lambdalimit}
For each $(s,t) \in [0,T]^2_\star$, we have

\begin{equation}\label{asZlimit}
\Lambda(s,t):=\lim_{(\epsilon,\delta)\downarrow 0}\Lambda^-(\epsilon,\delta;s,t) = \mathcal{W}(s,t) + \frac{\partial^2R}{\partial t\partial s}(s,t)I_{d\times d},
\end{equation}
almost surely, where


\begin{eqnarray}
\label{Wreplambda}\mathcal{W}^{ij}(s,t)&:=& \lambda_{11}(s,t)\lambda_{21}(s,t)\big(B^{(i)}_sB^{(j)}_s-\text{cov}(B^{(i)}_s; B^{(j)}_s)\big) \\
\nonumber &+& \lambda_{11}(s,t)\lambda_{22}(s,t) \big( B^{(i)}_s B^{(j)}_t - \text{cov}(B^{(i)}_s; B^{(j)}_t) \big)\\
\nonumber&+& \lambda_{12}(s,t) \lambda_{21}(s,t) \big( B^{(i)}_t B^{(j)}_s - \text{cov}(B^{(i)}_t; B^{(j)}_s) \big)\\
\nonumber&+& \lambda_{12}(s,t)\lambda_{22}(s,t)\big(B^{(i)}_t B^{(j)}_t - \text{cov}(B^{(i)}_t; B^{(j)}_t) \big),
\end{eqnarray}
and

$$\lambda_{11}(s,t) = \frac{1}{\Theta_{s,t}} \Big\{\frac{1}{2}v'(s) v(t) - R(s,t)\frac{\partial R}{\partial s} (s,t) \Big\},$$
$$\lambda_{12}(s,t)= \frac{1}{\Theta_{s,t}} \Big\{ \frac{\partial R}{\partial s} (s,t) v(s) - R(s,t)\frac{1}{2} v'(s) \Big\},$$
$$\lambda_{21}(s,t)= \frac{1}{\Theta_{s,t}} \Big\{\frac{\partial R}{\partial t} (s,t) v(t) - R(s,t)\frac{1}{2}v'(t) \Big\},$$
$$\lambda_{22}(s,t)=\frac{1}{\Theta_{s,t}} \Big\{\frac{1}{2}v'(t) v(s) - R(s,t)\frac{\partial R}{\partial t} (s,t) \Big\}$$
for $(s,t) \in [0,T]^2_\star$ and $1\le i,j\le d$.
\end{lemma}
\begin{proof}
Just apply Remark \ref{pointwiseR}, Lemma \ref{Sprerepr}, (\ref{Z1epsilon}), (\ref{Z2delta}) and use the pointwise differentiability of the covariance kernel $R$. In this case,

$$\lambda_{11}(s,t) = \lim_{\epsilon\downarrow 0}\frac{1}{\epsilon}\lambda_{11}(\epsilon,s,t),\quad \lambda_{12}(s,t)=\lim_{\epsilon\downarrow 0}\frac{1}{\epsilon}\lambda_{12}(\epsilon,s,t),$$
$$\lambda_{21}(s,t) = \lim_{\delta\downarrow 0}\frac{1}{\delta}\lambda_{21}(\delta,s,t), \quad \lambda_{22}(s,t) = \lim_{\delta\downarrow 0}\frac{1}{\delta}\lambda_{22}(\delta,s,t),$$
for each $(s,t) \in [0,T]^2_\star$.
\end{proof}


In the sequel, let us denote

$$\eta_{11}(s,t):=\lambda_{11}(s,t) + \lambda_{12}(s,t), \eta_{12}(s,t):= \lambda_{12}(s,t)$$
$$\eta_{21}(s,t):=\lambda_{21}(s,t) + \lambda_{22}(s,t), \eta_{22}(s,t):= \lambda_{22}(s,t),$$
for $(s,t) \in [0,T]^2_\star$. By using the fact that $\lambda_{11}(s,t) = \lambda_{22}(t,s)$ and $\lambda_{12}(s,t) = \lambda_{21}(t,s)$ for every $(s,t) \in [0,T]^2_\star$, we can also represent

\begin{eqnarray}
\label{INTREP1}\mathcal{W}^{ij}(s,t) &=& \Big(\eta_{11}(s,t)B^{(i)}_s + \eta_{12}(s,t)B^{(i)}_{s,t} \Big)\Big(\eta_{21}(s,t)B^{(j)}_s + \eta_{22}(s,t) B^{(j)}_{s,t} \Big)\\
\nonumber&-& \mathbb{E}\Big[\Big(\eta_{11}(s,t)B^{(i)}_s + \eta_{12}(s,t)B^{(i)}_{s,t} \Big)\Big(\eta_{21}(s,t)B^{(j)}_s + \eta_{22}(s,t) B^{(j)}_{s,t} \Big)\Big],
\end{eqnarray}
for $0< s < t\le T$ and

\begin{eqnarray}
\label{INTREP2}\mathcal{W}^{ij}(s,t) &=& \Big(\eta_{11}(t,s)B^{(j)}_t + \eta_{12}(t,s)B^{(j)}_{t,s} \Big)\Big(\eta_{21}(t,s)B^{(i)}_t + \eta_{22}(t,s) B^{(i)}_{t,s} \Big)\\
\nonumber&-& \mathbb{E}\Big[\Big(\eta_{11}(t,s)B^{(j)}_t + \eta_{12}(t,s)B^{(j)}_{t,s} \Big)\Big(\eta_{21}(t,s)B^{(i)}_t + \eta_{22}(t,s) B^{(i)}_{t,s} \Big)\Big],
\end{eqnarray}
for $0< t < s\le T$. The appearance of the increments $B_{s,t}$ and $B_{t,s}$ and the more regular functions $\eta_{11}$ and $\eta_{21}$ in representations (\ref{INTREP1}) and (\ref{INTREP2}) will be more convenient in the proof of Theorem \ref{mainTH1}.

\subsection{The projection operator over the 2-simplex}
The above discussion motivates us to view  $\Lambda^-(\epsilon,\delta; \cdot)$ in terms of increments $B_{s,t}$ rather than $B_s, B_t$ as follows. In case $0< s <t$, we will express \eqref{Lambdaop}
as

\begin{equation}\label{side1}
\Lambda^-(\epsilon,\delta; s,t)= \frac{1}{\epsilon\delta}\mathbb{E}\big[ B_{s,s+\epsilon}\otimes B_{t,t+\delta}| B_s,B_{s,t}\big].
\end{equation}
In case $0< t < s$, we will write

\begin{equation}\label{side2}
\Lambda^-(\epsilon,\delta; s,t)= \frac{1}{\epsilon\delta}\mathbb{E}\big[ B_{s,s+\epsilon}\otimes B_{t,t+\delta}| B_t,B_{t,s}\big].
\end{equation}
For this reason, we collect some important objects of the projection operator over $\Delta_T=\{(s,t) \in [0,T]^2_\star; 0< s < t\le T\}$ which will play a key role in our analysis. In this section, we restrict the discussion to the 2-simplex $\Delta_T$. The representation of (\ref{side2}) over the other $2$-simplex of $[0,T]^2_\star $ is totally analogous. Let us consider

$$
\mathcal{O}^{\epsilon,\delta}_{s,t}:=\left(
    \begin{array}{cc}
      \mathbf{o}_{11}(\epsilon,s) & \mathbf{o}_{12}(\epsilon,s,t) \\
      \mathbf{o}_{21}(\delta,s,t) &\mathbf{o}_{22}(\delta,s,t) \\
    \end{array}
  \right),
$$
where the elements $\mathbf{o}_{11}, \mathbf{o}_{12}, \mathbf{o}_{21}, \mathbf{o}_{22}$ are given by

$$\mathbf{o}_{11}(\epsilon,s):=\text{cov}\big(B^{(1)}_{s, s+\epsilon}, B^{(1)}_s \big)= R(s,s+\epsilon) - R(s,s),$$

\begin{eqnarray*}
\mathbf{o}_{12}(\epsilon,s,t)&:=&\text{cov}\big(B^{(1)}_{s, s+\epsilon}; B^{(1)}_{s,t}\big)=R(s+\epsilon,t) - R(s,t)\\
& - & [R(s+\epsilon,s) - R(s,s)],
\end{eqnarray*}

$$\mathbf{o}_{21}(\delta,s,t):=\text{cov}\big(B^{(1)}_{t, t+\delta}; B^{(1)}_s\big) = R(s,t+\delta) - R(s,t) $$
and

\begin{eqnarray*}
\mathbf{o}_{22}(\delta,s,t)&:=&\text{cov}\big(B^{(1)}_{t, t+\delta}; B^{(1)}_{s,t}\big)=R(t+\delta,t) - R(t,t)\\
 &-& [R(s,t+\delta) - R(s,t)],
 \end{eqnarray*}
for $(\epsilon,\delta) \in (0,1)^2$ and $(s,t) \in \Delta_T$. Gaussian linear regression yields the following representation

\begin{equation}\label{Z1epsINCR}
Z^{1;ij}_{s,t}(\epsilon) = \eta_{11}(\epsilon,s,t)B^{(i)}_s + \eta_{12}(\epsilon,s,t)B^{(i)}_{s,t}
 \end{equation}
and

\begin{equation}\label{Z2deINCR}
Z^{2;ij}_{s,t}(\delta)= \eta_{21}(\delta,s,t)B^{(j)}_s
+ \eta_{22}(\delta,s,t)B^{(j)}_{s,t},
\end{equation}
where we set

$$\eta_{11}(\epsilon,s,t):=\frac{1}{\Theta_{s,t}} \Big\{\mathbf{o}_{11}(\epsilon,s)|t-s|^{2H} - \mathbf{o}_{12}(\epsilon,s,t)\varphi(s,t)\Big\},$$
$$\eta_{12}(\epsilon,s,t):=\frac{1}{\Theta_{s,t}} \Big\{\mathbf{o}_{12}(\epsilon,s,t)s^{2H} - \mathbf{o}_{11}(\epsilon,s)\varphi(s,t)\Big\},$$
$$\eta_{21}(\delta,s,t):=\frac{1}{\Theta_{s,t}} \Big\{\mathbf{o}_{21}(\delta,s,t)|t-s|^{2H} - \mathbf{o}_{22}(\delta,s,t)\varphi(s,t)\Big\},$$
$$\eta_{22}(\delta,s,t):=\frac{1}{\Theta_{s,t}} \Big\{\mathbf{o}_{22}(\delta,s,t) s^{2H} - \mathbf{o}_{21}(\delta,s,t)\varphi(s,t)\Big\},$$
for $(s,t) \in \Delta_T$.
By identification with \eqref{Z1epsilon}  and \eqref{Z2delta}
one easily gets  $\eta_{11}(\epsilon,s,t) = \lambda_{11}(\epsilon,s,t) + \lambda_{12}(\epsilon,s,t)$, $\eta_{21}(\delta,s,t) = \lambda_{21}(\delta,s,t) + \lambda_{22}(\delta,s,t)$, $\eta_{12}(\epsilon,s,t) = \lambda_{12}(\epsilon,s,t)$ and $\eta_{22}(\delta,s,t) = \lambda_{22}(\delta,s,t)$, for $(s,t) \in \Delta_T$ and $(\epsilon,\delta) \in (0,1)^2$.

In the sequel, it will be convenient to introduce the following functions:

$$\vartheta_1(x):= |1+x|^{2H}-1 - |x|^{2H},$$

$$\vartheta_2(x):= |1+x|^{2H}-1;~x\in \mathbb{R}.$$

One can easily check we can write

$$\mathcal{O}^{\epsilon,\delta}_{s,t} = \left(
                                      \begin{array}{cc}
                                        \frac{s^{2H}}{2} \vartheta_1\big( \frac{\epsilon}{s}\big) & -\frac{|t-s|^{2H}}{2} \vartheta_1\big( \frac{-\epsilon}{t-s}\big) \\
                                         \frac{t^{2H}}{2} \vartheta_2\big( \frac{\delta}{t}\big)  -\frac{|t-s|^{2H}}{2} \vartheta_2\big( \frac{\delta}{t-s}\big)& \frac{|t-s|^{2H}}{2} \vartheta_1\big( \frac{\delta}{t-s}\big) \\
                                      \end{array}
                                    \right),
$$
for $(s,t)\in \Delta_T$. Then, 

\begin{equation}\label{eta11}
\Theta_{s,t}\eta_{11}(\epsilon,s,t) = \frac{s^{2H}}{2}\vartheta_1\Big(\frac{\epsilon}{s}\Big)|t-s|^{2H} + \frac{|t-s|^{2H}}{2}\varphi(s,t) \vartheta_1 \Big( -\frac{\epsilon}{t-s}\Big),
\end{equation}

\begin{equation}\label{eta12}
\Theta_{s,t}\eta_{12}(\epsilon,s,t) = -\frac{s^{2H}}{2}\vartheta_1 \Big( \frac{\epsilon}{s}\Big)\varphi(s,t) - \frac{|t-s|^{2H}}{2} s^{2H} \vartheta_1 \Big( -\frac{\epsilon}{t-s} \Big),
\end{equation}

\begin{eqnarray}
\nonumber\Theta_{s,t}\eta_{21}(\delta,s,t) &=& \Big\{\frac{t^{2H}}{2} \vartheta_2 \Big( \frac{\delta}{t} \Big) - \frac{|t-s|^{2H}}{2}\vartheta_2 \Big( \frac{\delta}{t-s}\Big)\Big\}|t-s|^{2H} \\
\label{eta21}&-& \varphi(s,t) \frac{|t-s|^{2H}}{2} \vartheta_1 \Big(\frac{\delta}{t-s}\Big),
\end{eqnarray}

\begin{eqnarray}
\nonumber\Theta_{s,t}\eta_{22}(\delta,s,t) &=& -\Big\{\frac{t^{2H}}{2} \vartheta_2 \Big( \frac{\delta}{t} \Big) - \frac{|t-s|^{2H}}{2}\vartheta_2 \Big( \frac{\delta}{t-s}\Big)\Big\}\varphi(s,t) \\
\label{eta22}&+& s^{2H}\frac{|t-s|^{2H}}{2} \vartheta_1 \Big(\frac{\delta}{t-s}\Big),
\end{eqnarray}
for $(s,t)\in \Delta_T$. Recall that $\frac{1}{2} < H < 1$ implies 

$$\varphi(s,t)>0;\quad (s,t) \in \Delta_T.$$
\begin{remark}\label{phi1rem}
Since,

$$\vartheta_2(x)= \vartheta_1(x) + |x|^{2H},$$
we can write $\mathcal{O}^{\epsilon,\delta}_{s,t}$ only in terms of $\vartheta_1$. Indeed, we notice that

$$\frac{t^{2H}}{2} \vartheta_2\Big( \frac{\delta}{t}\Big)  -\frac{|t-s|^{2H}}{2} \vartheta_2\Big( \frac{\delta}{t-s}\Big)$$
$$ = \Bigg[\frac{t^{2H}}{2} \vartheta_1\Big( \frac{\delta}{t}\Big)  -\frac{|t-s|^{2H}}{2} \vartheta_1\Big( \frac{\delta}{t-s}\Big)\Bigg],$$
for $(s,t) \in \Delta_T$. 
\end{remark}
Next, we recall the following important Lemma 3.4 in \cite{OhashiRusso2}.  
\begin{lemma}\label{fhlemma}
Fix $H  \in (0,1)$. The following representation holds:

\begin{equation}\label{detrep}
\Theta_{s,t} = |t-s|^{2H}A(s,t),
\end{equation}
where
$$
A(s,t):=\frac{1}{4}\Big\{ 2s^{2H}+ 2t^{2H} - |t-s|^{2H} - \frac{(t^{2H}-s^{2H})^2}{|t-s|^{2H}} \Big\},
$$
for $(s,t) \in [0,T]^2_\star$. Moreover,

$$\frac{s^{2H} \wedge t^{2H}}{|A(s,t)|} \le \frac{2^{2-2H}}{4-2^{2H}};~(s,t) \in [0,T]^2_\star.$$
Therefore,

\begin{equation}\label{deterest}
\frac{1}{\Theta^2_{s,t}}\lesssim_H \frac{1}{|t-s|^{4H}(s^{4H}\wedge t^{4H})},
\end{equation}
for every $(s,t) \in [0,T]^2_\star$.
\end{lemma}

\begin{lemma}\label{growth}
If $\frac{1}{2} < H < 1$, then
$$|\vartheta_1(x)|\le 4H |x|,$$
for every $x \in \mathbb{R}$.


\end{lemma}
\begin{proof}
At first, we consider the case $-1 \le x \le 1$. We observe $\vartheta_1$ is smooth except on the origin. The first derivative is

$$\vartheta'_1(x) = 2H \big[ (1+x)^{2H-1}-x^{2H-1}\big]; 0 < x\le 1$$
and

$$\vartheta'_1(x) = 2H \big[ (1+x)^{2H-1}+ (-x)^{2H-1}\big]; -1 \le x< 0.$$
Then, $|\vartheta'_1(x)|\le \max\{4H, H2^{2H}\}$ for every $-1\le x \le 1$ with $x\neq 0$.

We observe $\vartheta_1(x)>0$ for $x>1$ and $\vartheta_1(x) < 0$ for $x < -1$. Next, we consider the case $x>1$. Applying mean value theorem along the interval $(1, 1 + x^{-1})$, we have

$$\Big[ \Big(\frac{1+x}{x}\Big)^{2H}-1\Big] = \Big[ \Big(1+\frac{1}{x}\Big)^{2H}-1^{2H}\Big]\le 2H \Big(1+\frac{1}{x}\Big)^{2H-1}\frac{1}{x},$$
so that

\begin{eqnarray*}
\frac{\vartheta_1(x)}{2Hx} &=&\frac{1}{2H}\Big[ \Big(\frac{1+x}{x}\Big)^{2H} x^{2H-1}-\frac{1}{x} -x^{2H-1}  \Big]\\
&\le& \frac{1}{2H}x^{2H-1} \Big[ \Big(\frac{1+x}{x}\Big)^{2H}-1\Big]\\
&\le& (x+1)^{2H-1}\frac{1}{x} = \Big(\frac{x+1}{x}\Big)^{2H-1}x^{2H-2}.
\end{eqnarray*}
Therefore,

$$\Big|\frac{\vartheta_1(x)}{2Hx}\Big|\le 2^{2H-1},$$
for every $x>1.$ The analysis for $x < -1$ is similar. For sake of completeness, we give the details. We have $|\vartheta_1(x)| = (-x)^{2H}+1 - (-1-x)^{2H}$ for $x < -1$. Therefore,

$$\Big| \frac{\vartheta_1(x)}{2Hx}\Big|=\frac{1}{2H} \Big[ (-x)^{2H-1}+ \frac{1}{(-x)} - \frac{(-1-x)^{2H}}{(-x)}  \Big]; x <-1.$$
Consider the function
$$h(y)= \frac{1}{2H} \Big[ y^{2H-1}+ \frac{1}{y} - \frac{(y-1)^{2H}}{y}  \Big]; y >1.$$
Applying mean value theorem along the interval $(1-y^{-1}, 1)$, we have

$$\Big[ 1^{2H}-\Big(1-y^{-1}\Big)^{2H}\Big] \le 2H y^{-1},$$
so that

\begin{eqnarray*}
2H h(y)&=&  \Big[ 1^{2H}-\Big(1-y^{-1}\Big)^{2H}\Big]y^{2H-1} +\frac{1}{y}\\
&\le& 2Hy^{2H-2} + \frac{1}{y}\\
&\le& 2H+1,
\end{eqnarray*}
for every $y>1$. Therefore,

$$\Big| \frac{\vartheta_1(x)}{2Hx}\Big|\le 1 + \frac{1}{2H},$$
for every $x <-1$. Since $2H+1 < 4H$ and $H2^{2H} < 4H$ for every $\frac{1}{2} < H < 1$, we conclude the proof.

\end{proof}
\section{Proof of Theorems \ref{mainTH1} and \ref{mainTH2}}
This section presents the proofs of Theorems \ref{mainTH1} and \ref{mainTH2}. We split the matrix $\Lambda^-(\epsilon,\delta;\cdot)$ according to Lemma \ref{Sprerepr} as
\begin{eqnarray}
\label{ags}\Lambda^{-,ij}(\epsilon,\delta; s,t)&=& \frac{1}{\epsilon \delta}\Big[ Z^{1;ij}_{s,t}(\epsilon) Z^{2;ij}_{s,t}(\delta) - \mathbb{E}[
Z^{1;ij}_{s,t}(\epsilon) Z^{2;ij}_{s,t}(\delta)] \Big]\\
\nonumber&+& \frac{1}{\epsilon \delta}\mathbb{E}\big[B^{(i)}_{s,s+\epsilon} B^{(j)}_{t,t+\delta}\big],
\end{eqnarray}
for $(s,t) \in [0,T]^2_\star$ and $1\le i,j\le d$. In view of (\ref{INTREP1}) and (\ref{INTREP2}) (that is symmetry), it is sufficient to prove

$$\lim_{(\epsilon,\delta)\downarrow 0}\Lambda^-(\epsilon,\delta; \cdot) = \mathcal{W} + \frac{ \partial^2 R}{\partial t \partial s}I_{d\times d}$$
in $L^q(\Delta_T\times \Omega; \mathbb{R}^{d\times d})$ for $1 \le q < \min\Big\{ \frac{1}{2H-1}; \frac{1}{2-2H}  \Big\}$. The analysis over the other simplex of $[0,T]^2_*$ is entirely analogous. We will divide the proof into two separate parts:
$$\frac{1}{\epsilon \delta}\mathbb{E}\big[B^{(i)}_{s,s+\epsilon} B^{(j)}_{t,t+\delta}\big]\quad \text{(deterministic component)}$$ 
and 
$$\frac{1}{\epsilon \delta}\Big[ Z^{1;ij}_{s,t}(\epsilon) Z^{2;ij}_{s,t}(\delta) - \mathbb{E}[
Z^{1;ij}_{s,t}(\epsilon) Z^{2;ij}_{s,t}(\delta)] \Big]\quad \text{(random component)}$$
in (\ref{ags}).  
\subsection{The analysis of the deterministic component}
If $(s,t) \in \Delta_T$, then we can express

\begin{eqnarray}
\nonumber\frac{1}{\epsilon\delta}\mathbb{E}\big[ B^{(i)}_{s,s+\epsilon}B^{(j)}_{t,t+\delta} \big] &=& \frac{\delta_{ij}}{2\epsilon\delta}\Big\{ |t-s-\epsilon|^{2H}+ |t-s+\delta|^{2H} - |t-s + \delta-\epsilon|^{2H}  - |t-s|^{2H}\Big\}\\
 &=& \delta_{ij}\frac{|t-s|^{2H}}{2\epsilon\delta}\Phi\Big(\frac{\delta}{t-s}, \frac{\epsilon}{t-s}\Big)\label{Detrep},
 \end{eqnarray}
where 
$$\Phi(x,y) = |1+x|^{2H}  - \big|1+(x-y)\big|^{2H}-1+ |1-y|^{2H}; \quad (x,y) \in \mathbb{R}^2. $$

\begin{lemma}\label{grdet}
$$\sup_{ (x,y) \in Q }\Big|\frac{1}{xy}\Phi(x,y)\Big| < \infty,$$
where $Q=\{(x,y)\in \mathbb{R}_+^2; 0 < x\vee y < 1~\text{or}~0 < x\wedge y < 1 \le x\vee y\}$.
\end{lemma}
\begin{proof}
We set $f(x,y) = |1+x|^{2H}  - \big|1+(x-y)\big|^{2H}$. By Taylor's theorem, we observe for each $(x,y) \in \mathbb{R}^2_+$ such that $0 < x\wedge y \le x\vee y < 1$, we have

$$\Phi(x,y)= f(x,y)  - f(0,y) = \frac{\partial f}{\partial x}(\bar{x},y)x$$
for some $\bar{x}$ satisfying $0 < \bar{x} < x$. Here,

$$ \frac{\partial f}{\partial x}(x,y) = 2H \Big[(1+x)^{2H-1} - \big(1+(x-y)\big)^{2H-1} \Big],$$
for $0 < x\wedge y \le x\vee y < 1$. We observe $a \mapsto \Big[(1+a)^{2H-1} - \big(1+(a-y)\big)^{2H-1} \Big]$ is a positive decreasing function on $[0,1]$ for each $0 < y < 1 $. Then,

$$\max_{0\le a \le 1} \Big[(1+a)^{2H-1}- (1+(a-y))^{2H-1}\Big]= \Big[1-(1-y)^{2H-1}\Big],$$
for each $0 < y < 1$. Therefore, we can safely state

$$
\Big|\frac{1}{xy}\Phi(x,y)\Big| \le  2H \Bigg[\frac{1 - \big(1-y\big)^{2H-1}}{y}\Bigg],
$$
for $0 < x\wedge y \le x\vee y < 1$. By L'Hospital rule, we have

$$\lim_{y\downarrow 0}\Bigg[\frac{1 - \big(1-y\big)^{2H-1}}{y}\Bigg] = (2H-1)\lim_{y\downarrow 0}(1-y)^{2H-2} = 2H-1.$$

Now, we check
$$\sup_{0 < x < 1 < y}\Big|\frac{1}{xy}\Phi(x,y)\Big| < \infty.$$
and

$$\sup_{0 < y < 1 < x}\Big|\frac{1}{xy}\Phi(x,y)\Big| < \infty.$$
We observe we can always write

$$\Big|\frac{1}{xy}\Phi(x,y)\Big| = \frac{1}{xy}|\Phi(x,y) - \Phi(0,y)|$$
and

$$\Big|\frac{1}{xy}\Phi(x,y)\Big|= \frac{1}{xy}|\Phi(x,y) - \Phi(x,0)|,$$
for $x,y >0$. We observe there exists a constant $C$ which only depends on $H$ such that

\begin{eqnarray*}
\Big|\frac{\partial\Phi}{\partial y}(x,y)\Big| &=& \Big|2H\Big[ (1+x-y)^{2H-1} - (1-y)^{2H-1}\Big]\Big|\\
&\le& C 2H |x|^{2H-1}\\
&\le& C2H x,
\end{eqnarray*}
for $0\le y < 1 < x$. Then,

$$\Big|\frac{1}{xy}\Phi(x,y)\Big|\le C2H,$$
for every $0 < y < 1 < x$. Finally, we observe

$$\frac{\partial\Phi}{\partial x}(x,y) = 2H \Big[ (1+x)^{2H-1}-(1+x-y)^{2H-1}\Big],$$
for $0 < x < 1 < y\le 1+x$ and

$$\frac{\partial\Phi}{\partial x}(x,y) = 2H \Big[ (1+x)^{2H-1}+ (y-x-1)^{2H-1}\Big],$$
for $0 < x < 1+x < y$. Since $2H-1 < 1$, we observe there exists a constant $\bar{C}$ which only depends on $H$ such that

$$\big|\frac{\partial\Phi}{\partial x}(x,y)\big|\le 2H\bar{C}y,$$
for every $0 < x < 1 < y$. This concludes the proof.
\end{proof}

\begin{lemma}\label{diag}

If $1\le p < \frac{1}{2-2H}$, then
$$\int_{\Delta^1_T(\epsilon,\delta)}\Big|\frac{1}{\epsilon\delta}\mathbb{E}[B^{(i)}_{s,s+\epsilon}B^{(j)}_{t,t+\delta}]\Big|^pdsdt \lesssim (\epsilon \wedge \delta)^{2H + \frac{1}{p} -2}\rightarrow 0,$$
as $(\epsilon,\delta)\downarrow 0$, for every $1\le i,j\le d$. Here, 

$$\Delta^1_T(\epsilon,\delta) = \{(s,t) \in \Delta_T; (t-(\epsilon\wedge \delta))^+< s \le t\}\cup \{(s,t) \in \Delta_T;  s\le \epsilon \wedge \delta\}.$$
\end{lemma}
\begin{proof}
Fix $i\ne j$. Jensen and H\"{o}lder's inequality yield
\begin{eqnarray*}
\int_0^T \int_{(t-(\epsilon\wedge \delta))^+}^t \Big|\frac{1}{\epsilon\delta}\mathbb{E}[B^{(i)}_{s,s+\epsilon}B^{(j)}_{t,t+\delta}]\Big|^pdsdt&\le& \frac{1}{(\epsilon\delta)^p}\mathbb{E}|B^{(i)}_{s,s+\epsilon}B^{(j)}_{t,t+\delta}|^p \int_0^T \int_{(t-(\epsilon\wedge \delta))^+}^t dsdt\\
&=& \Big(\frac{\epsilon^{H} \delta^{H}}{\epsilon \delta}\Big)^p \int_0^T \{t - (t-(\epsilon\wedge \delta))^+\}dt\\
&\lesssim& \Big(\frac{\epsilon^{H} \delta^{H}}{\epsilon \delta}\Big)^p  (\epsilon\wedge \delta)
\end{eqnarray*}
where

$$
\frac{\epsilon^H \delta^{H}}{\epsilon \delta} (\delta \wedge \epsilon)^{\frac{1}{p}}\lesssim\left\{
\begin{array}{rl}
\delta^{2H + \frac{1}{p}-2}; & \hbox{if} \ \delta < \epsilon \\
\epsilon^{2H + \frac{1}{p}-2};& \hbox{if} \ \delta \ge \epsilon. \\
\end{array}
\right.
$$
In other words, $\frac{\epsilon^H \delta^{H}}{\epsilon \delta} (\delta \wedge \epsilon)^{\frac{1}{p}}\lesssim (\epsilon \wedge \delta)^{2H + \frac{1}{p} -2}\rightarrow 0$ as $(\delta,\epsilon) \downarrow 0$. Similarly,

\begin{eqnarray*}
\mathbb{E}\int_0^T \int_{0}^{\epsilon\wedge \delta} \Big|\frac{1}{\epsilon\delta}\mathbb{E}[B^{(i)}_{s,s+\epsilon}B^{(j)}_{t,t+\delta}]\Big|^pdsdt&\le& \frac{1}{(\epsilon\delta)^p}\mathbb{E}|B^{(i)}_{s,s+\epsilon}B^{(j)}_{t,t+\delta}|^p \int_0^T \int_{0}^{\epsilon\wedge \delta}ds dt\\
&\lesssim& \Big(\frac{\epsilon^{H} \delta^{H}}{\epsilon \delta}\Big)^p  (\epsilon\wedge \delta)\rightarrow 0,
\end{eqnarray*}
as $(\epsilon,\delta)\downarrow 0$.
\end{proof}

\begin{proposition}\label{deterProp}
For each $i,j = 1,\ldots, d$, we have 
$$\int_{[0,T]^2}\Big|\frac{1}{\epsilon\delta}\mathbb{E}\big[ B^{(j)}_{t,t+\delta} B^{(i)}_{s,s+\epsilon} \big] - \frac{\partial^2R}{\partial s\partial t}(s,t)\delta_{ij}  \Big|^p dsdt\rightarrow 0,$$
as $(\epsilon,\delta)\downarrow 0$ for $1\le p < \frac{1}{2-2H}$.
\end{proposition}
\begin{proof}
Fix $i=j$. By symmetry, it is sufficient to prove

$$\int_{\Delta_T}\Big|\frac{1}{\epsilon\delta}\mathbb{E}\big[ B^{(j)}_{t,t+\delta} B^{(i)}_{s,s+\epsilon} \big] - \frac{\partial^2R}{\partial s\partial t}(s,t) \Big|^p dsdt\rightarrow 0,$$
as $(\epsilon,\delta)\downarrow 0$ for $1\le p < \frac{1}{2-2H}$. Fix $(\epsilon,\delta) \in (0,1)^2$. We decompose

\begin{eqnarray*}
\Delta_T &=& \{(s,t) \in \Delta_T; 1\le (\epsilon\wedge \delta)(t-s)^{-1}\le (\epsilon\vee \delta)(t-s)^{-1}\}\\
&\cup&  \{(s,t) \in \Delta_T; 0< (\epsilon\wedge \delta)(t-s)^{-1} < 1 \le (\epsilon\vee \delta)(t-s)^{-1} \}\\
&\cup& \{(s,t) \in \Delta_T; 0< (\epsilon\wedge \delta)(t-s)^{-1}\le (\epsilon\vee \delta)(t-s)^{-1} < 1 \}\\
&=:&A_1(\epsilon,\delta)\cup A_2(\epsilon,\delta) \cup A_3(\epsilon,\delta).
\end{eqnarray*}
By Lemma \ref{diag}, if $1\le p < \frac{1}{2-2H}$, we have

$$\int_{A_1(\epsilon,\delta)}\Big|\frac{1}{\epsilon\delta}\mathbb{E}\big[ B^{(j)}_{t,t+\delta} B^{(i)}_{s,s+\epsilon} \big]\Big|^p dsdt\rightarrow 0,$$
as $\epsilon,\delta\downarrow 0$. By using (\ref{Detrep}) and Lemma \ref{grdet}, we observe

$$\int_{A_2(\epsilon,\delta)\cup A_3(\epsilon,\delta) }\Big|\frac{1}{\epsilon\delta}\mathbb{E}\big[ B^{(j)}_{t,t+\delta} B^{(i)}_{s,s+\epsilon} \big]\Big|^p dsdt\lesssim \int_{\Delta_T}|t-s|^{(2H-2)p}dsdt,$$
for $p\ge 1$ such that $(2H-2)p+1>0$. We do have pointwise convergence (see Remark \ref{pointwiseR})

$$\frac{1}{\epsilon\delta}\mathbb{E}\big[ B^{(j)}_{t,t+\delta} B^{(i)}_{s,s+\epsilon} \big]\rightarrow \frac{\partial^2 R}{\partial t\partial s}(s,t),$$
as $(\epsilon,\delta)\downarrow 0$, for each $(s,t) \in \Delta_T$ and hence, we conclude the proof by uniform integrability.
\end{proof}

\subsection{The analysis of the random component}

 In the sequel, we denote \\ $\mathcal{W}^-(\epsilon,\delta;s,t) = (\mathcal{W}^{i,j}(\epsilon,\delta;s,t))$,
where
$$\mathcal{W}^{-,ij}(\epsilon,\delta;s,t):= \frac{1}{\epsilon \delta}\Big[ Z^{1;ij}_{s,t}(\epsilon) Z^{2;ij}_{s,t}(\delta) - \mathbb{E}[
Z^{1;ij}_{s,t}(\epsilon) Z^{2;ij}_{s,t}(\delta)] \Big],$$
for $(s,t) \in [0,T]^2_\star$ and $1\le i,j\le d$. 

The aim of this section is devoted to the proof of the following result.
\begin{proposition}\label{Wpropo}
For $\frac{1}{2} < H < 1$ and $1\le q < \min \Big \{ \frac{1}{2H-1}; \frac{1}{2-2H}\Big\}$, we have  
$$\mathbb{E}\int_{[0,T]^2}| \mathcal{W}^-(\epsilon,\delta;s,t)   - \mathcal{W}(s,t)|^q dsdt\rightarrow 0,$$
as $(\epsilon,\delta)\downarrow 0$. 
\end{proposition}
Before starting the proof of Proposition \ref{Wpropo}, it is convenient to say a few words about the strategy. By Lemma \ref{Lambdalimit}, we know that 

$$\lim_{(\epsilon,\delta)\downarrow 0} \mathcal{W}^{-,ij}(\epsilon,\delta;s,t)= \mathcal{W}(s,t)$$
a.s. for each $(s,t) \in [0,T]^2_\star$. By using representations (\ref{Z1epsINCR}), (\ref{Z2deINCR}), (\ref{eta11}), (\ref{eta12}), (\ref{eta21}) and (\ref{eta22}), we have

\begin{eqnarray*} 
\mathcal{W}^{-,ij}(\epsilon,\delta;s,t)&=&\frac{1}{\epsilon\delta} \eta_{11}(\epsilon,s,t)\eta_{22}(\delta,s,t)\Big[ B^{(i)}_s B^{(j)}_{s,t} - \varphi(s,t)\delta_{ij} \Big]\\
&+&  \frac{1}{\epsilon\delta}\eta_{12}(\epsilon,s,t)\eta_{21}(\delta,s,t) \Big[ B^{(j)}_s B^{(i)}_{s,t} - \varphi(s,t)\delta_{ij} \Big] \\
&+&   \frac{1}{\epsilon\delta} \eta_{11}(\epsilon,s,t)\eta_{21}(\delta,s,t) \big[ B^{(i)}_s B^{(j)}_s  - s^{2H}\delta_{ij}\big]\\ 
&+& \frac{1}{\epsilon\delta} \eta_{12}(\epsilon,s,t) \eta_{22}(\delta,s,t)\Big[  B^{(i)}_{s,t} B^{(j)}_{s,t} - |t-s|^{2H}\delta_{ij}\Big]
\end{eqnarray*}
for $(s,t) \in \Delta_T$. In Lemmas \ref{lemmaf1}, \ref{lemmaf2} and \ref{lemmaf3}, we will show there exist $f_1, f_2, f_3, f_4 \in L^q(\Delta_T)$ for $1\le q < \min \Big \{ \frac{1}{2H-1}; \frac{1}{2-2H}\Big\}$ such that  

\begin{equation}\label{Exf1}
\Big|\frac{1}{\epsilon\delta} \eta_{11}(\epsilon,s,t)\eta_{22}(\delta,s,t)\Big| s^H |t-s|^H\lesssim_{H,T} f_1(s,t),
\end{equation}
\begin{equation}\label{Exf2}
\Big|\frac{1}{\epsilon\delta}\eta_{12}(\epsilon,s,t)\eta_{21}(\delta,s,t)\Big|s^H|t-s|^H\lesssim_{H,T} f_2(s,t),
\end{equation}
\begin{equation}\label{Exf3} 
\Big|\frac{1}{\epsilon\delta} \eta_{11}(\epsilon,s,t)\eta_{21}(\delta,s,t)\Big| s^{2H}\lesssim_{H,T} f_3(s,t),
\end{equation}
\begin{equation}\label{Exf4}
\Big|\frac{1}{\epsilon\delta} \eta_{12}(\epsilon,s,t) \eta_{22}(\delta,s,t)\Big||t-s|^{2H}\lesssim_{H,T} f_4(s,t),
\end{equation}
for every $(s,t)\in \Delta_T$ and $(\epsilon,\delta) \in (0,1)^2$. Observe that $\varphi(s,t)>0$ and 

\begin{equation}\label{Exf5}
\varphi(s,t)\lesssim_H \min \{s  t^{2H-1}, |t-s| t^{2H-1}, s^H|t-s|^H\},
\end{equation}
for $(s,t) \in \Delta_T$. Moreover,   

\begin{equation}\label{Exf6}
\sup_{(s,t) \in \Delta_T}\| B^{(i)}_s B^{(j)}_{s,t} s^{-H}|t-s|^{-H}\|^p_p + \sup_{(s,t) \in \Delta_T} \| B^{(i)}_s B^{(j)}_{s} s^{-2H}\|_p^p< \infty,
\end{equation} 

\begin{equation}\label{Exf7}
\sup_{(s,t) \in \Delta_T}\|B^{(i)}_{s,t} B^{(j)}_{s,t} |t-s|^{-2H}\|_p^p < \infty,
\end{equation}
for every $p> 1$ and $1\le i,j\le d$. The estimates (\ref{Exf1}), (\ref{Exf2}), (\ref{Exf3}), (\ref{Exf4}), (\ref{Exf5}), (\ref{Exf6}), (\ref{Exf7}) combined with the symmetry described in (\ref{INTREP1}) and (\ref{INTREP2}) will allow us to conclude the proof of Proposition \ref{Wpropo}. In order to shorten notation, we set

$$d(\delta,s,t):= \frac{t^{2H}}{2} \vartheta_1 \Big( \frac{\delta}{t} \Big) - \frac{|t-s|^{2H}}{2}\vartheta_1 \Big( \frac{\delta}{t-s}\Big) $$
for $(s,t)\in \Delta_T$. 
\begin{lemma}\label{ddeltaLEMMA}

$$
|d(\delta,s,t)|\lesssim_H \delta s (t^{-1}+1),$$
for $0 < s < t$. 
\end{lemma}

\begin{proof}
Mean value theorem and the fact that $2H-1>0$ imply 
$$\big(t^{2H} - (t-s)^{2H}\big)\le 2Ht^{2H-1}s,$$
whenever $0 < s < t$. Then, Lemma \ref{growth} and triangle inequality yield

$$ \Big|\frac{t^{2H}}{2} \vartheta_1 \Big( \frac{\delta}{t} \Big) - \frac{|t-s|^{2H}}{2}\vartheta_1 \Big( \frac{\delta}{t-s}\Big)\Big|$$
$$\lesssim \Big| \big(t^{2H} - (t-s)^{2H}\big) \vartheta_1 \Big( \frac{\delta}{t}\Big)   \Big| +  \big(t - s \big)^{2H} \Big|\vartheta_1 \Big( \frac{\delta}{t}\Big) - \vartheta_1 \Big( \frac{\delta}{t-s}\Big)   \Big|$$
$$
\lesssim s t^{2H-2} \delta + \big(t - s \big)^{2H} \Big|\vartheta_1 \Big( \frac{\delta}{t}\Big) - \vartheta_1 \Big( \frac{\delta}{t-s}\Big)   \Big|,
$$
for $0 < s<t$. The function $\vartheta_1$ is strictly increasing and hence by applying again Mean value theorem and using the fact that $2H-1>0$, we get

\begin{eqnarray*}
\Big|\vartheta_1 \Big( \frac{\delta}{t}\Big) - \vartheta_1 \Big( \frac{\delta}{t-s}\Big)   \Big|&=&\vartheta_1 \Big( \frac{\delta}{t-s}\Big) - \vartheta_1 \Big( \frac{\delta}{t}\Big)\\
&=&\Big(1+\frac{\delta}{t-s}\Big)^{2H} - \Big(1+\frac{\delta}{t}\Big)^{2H}  - \Big(\frac{\delta}{t-s}\Big)^{2H} + \Big(\frac{\delta}{t}\Big)^{2H}\\
&<& \Big(1+\frac{\delta}{t-s}\Big)^{2H} - \Big(1+\frac{\delta}{t}\Big)^{2H}\\
&\lesssim& \frac{\delta s}{t |t-s|}\Big( 1+ \frac{\delta}{t-s}\Big)^{2H-1}\\
&\lesssim&  \frac{\delta s}{t |t-s|} +  \frac{ s}{t |t-s|} \frac{\delta^{2H}}{|t-s|^{2H-1}},
\end{eqnarray*}
whenever $0 < s < t$. Then, 
\begin{eqnarray*}
|d(\delta,s,t)|&\lesssim_H& \delta st^{2H-2}+ \delta \frac{s}{t}|t-s|^{2H-1} + \delta^{2H}\frac{s}{t}\\
&\lesssim_H& \delta st^{2H-2} + \delta^{2H}s t^{-1}\\
&\lesssim_H& s\delta (t^{-1}+1), 
\end{eqnarray*}
for $0 < s < t$, because $2-2H<1$ and $2H>1$.





\end{proof}

\begin{lemma}\label{lemmaf1}
\begin{eqnarray}
\nonumber \Big|\frac{1}{\epsilon}\eta_{12}(\epsilon,s,t)\frac{1}{\delta} \eta_{22}(\delta,s,t) |t-s|^{2H}\Big|&\lesssim_{H,T}& t^{-1} + s^{H-1}|t-s|^{H-1} + s^{1-2H}(t^{2H-2}+1)\\ 
\label{T1}&+& |t-s|^{2H-2},
\end{eqnarray}
for every $(\epsilon,\delta) \in (0,1)^2$ and $(s,t) \in \Delta_T$. Moreover, the right-hand side of (\ref{T1}) is $q$-Lebesgue integrable over $\Delta_T$ for every $1\le q < \min \Big\{\frac{1}{2H-1}, \frac{1}{2-2H}\Big\}$.
\end{lemma}

\begin{proof}

By using Remark \ref{phi1rem}, we write


\begin{eqnarray*}
\frac{1}{\epsilon}\eta_{12}(\epsilon,s,t)\frac{1}{\delta}\eta_{22}(\delta,s,t) |t-s|^{2H} &=& \frac{\varphi^2(s,t)|t-s|^{2H}}{\epsilon \delta \Theta^2_{s,t}}\frac{s^{2H}}{2}\vartheta_1 \Big( \frac{\epsilon}{s}\Big)d(\delta,s,t)\\
&-& \frac{|t-s|^{2H}}{\epsilon \delta \Theta^2_{s,t}}\frac{s^{2H}}{2}\vartheta_1 \Big( \frac{\epsilon}{s}\Big)\frac{s^{2H}}{2}|t-s|^{2H}\vartheta_1 \Big( \frac{\delta}{t-s}\Big)\varphi(s,t)\\
&+&\frac{|t-s|^{2H}}{\epsilon \delta \Theta^2_{s,t}}\varphi(s,t)d(\delta,s,t)\frac{|t-s|^{2H}}{2}s^{2H}\vartheta_1 \Big( \frac{-\epsilon}{t-s} \Big)\\
&-& \vartheta_1 \Big( \frac{-\epsilon}{t-s} \Big)\vartheta_1 \Big( \frac{\delta}{t-s} \Big)\Bigg(\frac{|t-s|^{2H}s^{2H}}{2}\Bigg)^2\frac{|t-s|^{2H}}{\epsilon \delta \Theta^2_{s,t}}\\
&=:& I_1(\epsilon,\delta,s,t)+I_2(\epsilon,\delta,s,t)+I_3(\epsilon,\delta,s,t)+I_4(\epsilon,\delta,s,t).
\end{eqnarray*}

By definition
$$
\Theta_{s,t}= |t-s|^{2H} A(s,t),
$$
where $A$ is the function defined in Lemma \ref{fhlemma}.
By Lemmas \ref{growth}, \ref{fhlemma} and (\ref{Exf5}), we observe

\begin{equation}\label{e4}
|I_4(\epsilon,\delta,s,t)|\lesssim |t-s|^{2H-2}
\end{equation}
and 

\begin{eqnarray}
\nonumber|I_2(\epsilon,\delta,s,t)|&\lesssim& |t-s|^{-1} s^{-1} |\varphi(s,t)|\\
\label{e5}&\lesssim& s^{H-1}|t-s|^{H-1},
\end{eqnarray}
for $(s,t)\in \Delta_T$. Similarly, Lemmas \ref{growth}, \ref{fhlemma}, \ref{ddeltaLEMMA} and (\ref{Exf5}) yield 

\begin{eqnarray*}
|I_1(\epsilon,\delta,s,t)|&\lesssim& \Big| \frac{d(\delta,s,t)}{\delta} \Big|\varphi^2(s,t) |t-s|^{-2H} s^{-2H-1}\\
&\lesssim& s(t^{-1}+1)\varphi^2(s,t) |t-s|^{-2H} s^{-2H-1}\\
&+& (t^{-1}+1)\varphi^2(s,t) |t-s|^{-2H} s^{-2H}\\
&\lesssim& t^{-1} +1,
\end{eqnarray*}
and
\begin{eqnarray*}
|I_3(\epsilon,\delta,s,t)|&\lesssim& \Big| \frac{d(\delta,s,t)}{\delta} \Big|\varphi (s,t) |t-s|^{-1} s^{-2H}\\
&\lesssim& s(t^{-1}+1)\varphi(s,t)|t-s|^{-1}s^{-2H}\\
&\lesssim& s^{1-2H} (t^{-1}+1)t^{2H-1}\\
&\lesssim& s^{1-2H}(t^{2H-2} +1),
\end{eqnarray*}
for $(s,t) \in \Delta_T$. This shows the estimate (\ref{T1}). It is straightforward to check that the right-hand side of (\ref{T1}) is $q$-Lebesgue integrable over $\Delta_T$ for every $1\le q < \min \Big\{\frac{1}{2H-1}, \frac{1}{2-2H}\Big\}$.
\end{proof}

\begin{lemma}\label{lemmaf2}
\begin{eqnarray}
\nonumber\big|\frac{1}{\epsilon}\eta_{11}(\epsilon,s,t)\frac{1}{\delta}\eta_{21}(\delta,s,t) s^{2H}\big|&\lesssim_{H,T}& (t^{-1}+1) + s^{H-1}|t-s|^{H-1} + s^{1-2H}(t^{2H-2}+1)\\ 
\label{T2}&+& |t-s|^{2H-2},
\end{eqnarray}
for every $(\epsilon,\delta) \in (0,1)^2$ and $(s,t) \in \Delta_T$. Moreover, the right-hand side of (\ref{T2}) is $q$-Lebesgue integrable over $\Delta_T$ for $1\le q < \min\Big\{\frac{1}{2-2H}; \frac{1}{2H-1} \Big\}$.
\end{lemma}

\begin{proof}
Recall (apply Remark \ref{phi1rem}),
$$\frac{1}{\epsilon}\eta_{11}(\epsilon,s,t)\frac{1}{\delta}\eta_{21}(\delta,s,t) s^{2H} = $$
$$d(\delta,s,t)|t-s|^{2H} \frac{s^{2H}}{2}\vartheta_1\Big(\frac{\epsilon}{s}\Big)|t-s|^{2H}\frac{s^{2H}}{\epsilon \delta \Theta^2_{s,t}}$$

$$+d(\delta,s,t)|t-s|^{2H} \frac{|t-s|^{2H}}{2} \varphi(s,t) \vartheta_1 \Big( -\frac{\epsilon}{t-s}\Big)\frac{s^{2H}}{\epsilon \delta \Theta^2_{s,t}}$$
$$-\frac{|t-s|^{2H}}{2} \vartheta_1 \Big(\frac{\delta}{t-s}\Big) \varphi(s,t) \times \frac{s^{2H}}{2}\vartheta_1\Big(\frac{\epsilon}{s}\Big)|t-s|^{2H} \frac{s^{2H}}{\epsilon \delta \Theta^2_{s,t}}  $$
$$ - \frac{|t-s|^{2H}}{2} \vartheta_1 \Big(\frac{\delta}{t-s}\Big) \varphi(s,t) \times \frac{|t-s|^{2H}}{2} \varphi(s,t) \vartheta_1 \Big( -\frac{\epsilon}{t-s}\Big) \frac{s^{2H}}{\epsilon \delta \Theta^2_{s,t}}$$
$$=: I_1(\epsilon,\delta,s,t) + I_2(\epsilon,\delta,s,t) + I_3(\epsilon,\delta,s,t)+ I_4(\epsilon,\delta,s,t).$$



Lemmas \ref{growth} and \ref{ddeltaLEMMA} yield 

\begin{eqnarray*}
|I_1(\epsilon,\delta,s,t)|&\lesssim& \Big|\frac{d(\delta,s,t)}{\delta}\Big|s^{-1}\\
&\lesssim& t^{-1} +1, 
\end{eqnarray*}
and

\begin{eqnarray*}
|I_2(\epsilon,\delta,s,t)|&\lesssim& \Big|\frac{d(\delta,s,t)}{\delta}\Big|(t-s)^{-1}s^{-2H}\varphi(s,t)\\
&\lesssim& s^{1-2H}(t^{-1} +1) (t-s)^{-1}\varphi(s,t)\\
&\lesssim& s^{1-2H}(t^{-1} +1)t^{2H-1}\\
&\lesssim& s^{1-2H}(t^{2H-2}+1),
\end{eqnarray*}
for $(s,t)\in \Delta_T$. From Lemmas \ref{fhlemma}, \ref{growth} and (\ref{Exf5}), we also have

\begin{eqnarray*}
|I_3(\epsilon,\delta,s,t)|&\lesssim& | t-s|^{4H}\Big| \vartheta_1 \Big(\frac{\delta}{t-s}\Big) \vartheta_1 \Big(\frac{\epsilon}{s}\Big)\Big| \varphi(s,t) s^{4H}\frac{1}{\epsilon\delta \Theta^2_{s,t}}\\
&\lesssim& \frac{(t-s)^{4H-1} \varphi(s,t) s^{4H-1}}{(t-s)^{4H}s^{4H}}\\
&=& s^{-1}|t-s|^{-1}\varphi(s,t)\\
&\lesssim& s^{H-1}|t-s|^{H-1},
\end{eqnarray*}
and

$$|I_4(\epsilon,\delta,s,t)|\lesssim |t-s|^{2H-2},$$
for $(s,t) \in \Delta_T.$ This concludes the proof.
\end{proof}

\begin{lemma}\label{lemmaf3}
\begin{eqnarray}
\Big|\frac{1}{\epsilon} \eta_{11}(\epsilon,s,t)\frac{1}{\delta} \eta_{22}(\delta,s,t)s^H(t-s)^H\Big|&\lesssim_{H,T}& s^{H-1}|t-s|^{H-1} + \label{i11i22}|t-s|^{2H-2}\\
\nonumber&+& (t^{-1}+1) + |t-s|^{H-1}(1+ t^{-H}),
\end{eqnarray}
and

\begin{eqnarray}
\Big|\frac{1}{\epsilon} \eta_{12}(\epsilon,s,t)\frac{1}{\delta} \eta_{21}(\delta,s,t)s^H(t-s)^H\Big|&\lesssim_{H,T}& s^{H-1}|t-s|^{H-1} + \label{i12i21}|t-s|^{2H-2}\\
\nonumber&+& (t^{-1}+1) + |t-s|^{H-1}(1+ t^{-H}),
\end{eqnarray}
for every $(\epsilon,\delta) \in (0,1)^2$ and $(s,t) \in \Delta_T$. Moreover, the right-hand side of (\ref{i11i22}) and (\ref{i12i21}) are $q$-Lebesgue integrable over $\Delta_T$ for $1\le q < \min \Big\{ \frac{1}{2-2H};2 \Big\}$.
\end{lemma}

\begin{proof}
By definition,

$$\frac{1}{\epsilon} \eta_{11}(\epsilon,s,t)\frac{1}{\delta} \eta_{22}(\delta,s,t) \times s^{H}(t-s)^{H} = $$
$$-\frac{\varphi(s,t)}{\Theta^2_{s,t}\epsilon\delta}d(\delta,s,t)\frac{s^{2H}}{2}\vartheta_1 \Big( \frac{\epsilon}{s}\Big)(t-s)^{2H}\times s^{H}(t-s)^{H}$$

$$+ \frac{\varphi(s,t)^2}{\Theta^2_{s,t}\epsilon\delta}d(\delta,s,t) \frac{(t-s)^{2H}}{2}\vartheta_1 \Big( \frac{-\epsilon}{t-s}\Big)\times s^{H}(t-s)^{H}$$

$$+ \frac{s^{2H}(t-s)^{2H}}{2\epsilon\delta \Theta^2_{s,t}}\vartheta_1 \Big( \frac{\delta}{t-s} \Big)\frac{s^{2H}}{2}\vartheta_1 \Big( \frac{\epsilon}{s} \Big)(t-s)^{2H}\times s^{H}(t-s)^{H}$$

$$+ \frac{s^{2H}(t-s)^{2H}}{2\epsilon\delta \Theta^2_{s,t}}\vartheta_1 \Big( \frac{\delta}{t-s} \Big)\frac{(t-s)^{2H}}{2} \vartheta_1 \Big( \frac{-\epsilon}{t-s} \Big) \varphi(s,t)\times s^{H}(t-s)^{H}$$
$$ =: I_1(\epsilon,\delta,s,t) + I_2(\epsilon,\delta,s,t) + I_3(\epsilon,\delta,s,t) + I_4(\epsilon,\delta,s,t).$$

Lemmas \ref{growth}, \ref{fhlemma} and (\ref{Exf5}) yield 

$$|I_3(\epsilon,\delta,s,t)|\lesssim s^{H-1} |t-s|^{H-1}$$
and 

\begin{eqnarray*}
|I_4(\epsilon,\delta,s,t)|&\lesssim& s^{-H}|t-s|^{H-2}\varphi(s,t)\\
&\lesssim& |t-s|^{2H-2},
\end{eqnarray*}
for $(s,t) \in \Delta_T$.  Lemmas \ref{growth}, \ref{fhlemma}, \ref{ddeltaLEMMA} and (\ref{Exf5}) yield

\begin{eqnarray*}
|I_1(\epsilon,\delta,s,t)|&\lesssim& \Big|\frac{d(\delta,s,t)}{\delta}\Big| \varphi(s,t) s^{-H-1}|t-s|^{-H}\\
&\lesssim& s^{-H}|t-s|^{-H} \varphi(s,t)(t^{-1} +1)\\
&\lesssim& (t^{-1}+1),
\end{eqnarray*}
and
\begin{eqnarray*}
|I_2(\epsilon,\delta,s,t)|&\lesssim& \Big|\frac{d(\delta,s,t)}{\delta}\Big| \varphi^2(s,t) s^{-3H}|t-s|^{-H-1}\\
&\lesssim& s^{1-3H}|t-s|^{-H-1} \varphi^2(s,t)(t^{-1} +1)\\
&\lesssim& s^{1-H} |t-s|^{H-1} (t^{-1}+1)\\
&\lesssim& t^{-H}|t-s|^{H-1} + |t-s|^{H-1}, 
\end{eqnarray*}
for $(s,t) \in \Delta_T$. 

Next, we evaluate

$$\frac{1}{\epsilon} \eta_{12}(\epsilon,s,t)\frac{1}{\delta} \eta_{21}(\delta,s,t) \times s^{H}(t-s)^{H} = $$
$$-\frac{1}{2\Theta^2_{s,t}\epsilon\delta} s^{2H} \vartheta_1\Big(\frac{\epsilon}{s}\Big)\varphi(s,t) d(\delta,s,t)|t-s|^{2H}\times s^{H}(t-s)^{H}$$
$$+\frac{1}{2\Theta^2_{s,t}\epsilon\delta} s^{2H} \vartheta_1\Big(\frac{\epsilon}{s}\Big)\varphi(s,t)\frac{1}{2}|t-s|^{2H}\vartheta_1\Big(\frac{\delta}{t-s}\Big)\varphi(s,t)\times s^{H}(t-s)^{H}  $$

$$-\frac{1}{2\Theta^2_{s,t}\epsilon\delta} |t-s|^{2H} s^{2H} \vartheta_1\Big(\frac{-\epsilon}{t-s}\Big) d(\delta,s,t)|t-s|^{2H}\times s^{H}(t-s)^{H}$$

$$+\frac{1}{4\Theta^2_{s,t}\epsilon\delta} |t-s|^{2H} s^{2H} \vartheta_1\Big(\frac{-\epsilon}{t-s}\Big) |t-s|^{2H} \vartheta_1\Big(\frac{\delta}{t-s}\Big) \varphi(s,t)  \times s^{H}(t-s)^{H}$$

$$ =: J_1(\epsilon,\delta,s,t) + J_2(\epsilon,\delta,s,t) + J_3(\epsilon,\delta,s,t) + J_4(\epsilon,\delta,s,t),$$
for $(s,t) \in \Delta_T$. Lemmas \ref{fhlemma}, \ref{growth} and (\ref{Exf5}) yield 

\begin{eqnarray*}
|J_2(\epsilon,\delta,s,t)|&\lesssim& \varphi^2(s,t)s^{-H-1}|t-s|^{-H-1}\\
&\lesssim& s^{H-1}|t-s|^{H-1}
\end{eqnarray*}
and 

\begin{eqnarray*}
|J_4(\epsilon,\delta,s,t)|&\lesssim& \varphi(s,t) s^{-H}|t-s|^{H-2}\\
&\lesssim& |t-s|^{2H-2},
\end{eqnarray*}
for $(s,t) \in \Delta_T$. Lemmas \ref{fhlemma}, \ref{growth}, \ref{ddeltaLEMMA} and (\ref{Exf5}) yield 

\begin{eqnarray*}
|J_1(\epsilon,\delta,s,t)|&\lesssim& \Big|\frac{d(\delta,s,t)}{\delta}\Big|  \varphi(s,t) s^{-H-1}|t-s|^{-H}\\
&\lesssim& s^{-H}|t-s|^{-H}\varphi(s,t)(t^{-1} + 1)\\
&\lesssim& t^{-1} + 1,
\end{eqnarray*}
and 

\begin{eqnarray*}
|J_3(\epsilon,\delta,s,t)|&\lesssim& \Big|\frac{d(\delta,s,t)}{\delta}\Big| s^{-H}|t-s|^{H-1}\\
&\lesssim& s^{1-H}|t-s|^{H-1} (t^{-1} + 1)\\
&\lesssim& |t-s|^{H-1} (1 + t^{-H}),
\end{eqnarray*}
for $(s,t) \in \Delta_T$. Now, observe that $\frac{1}{2-2H} < \frac{1}{1-H}$ and hence the right-hand side of (\ref{i11i22}) and (\ref{i12i21}) are $q$-Lebesgue integrable over $\Delta_T$ for $1\le q < \min \Big\{ \frac{1}{2-2H};2 \Big\}$. 
\end{proof}

\textbf{Proof of Proposition \ref{Wpropo}:} Observe $\frac{1}{2H-1} < 2 < \frac{1}{2-2H}$ for $\frac{3}{4} < H<1$ and $\frac{1}{2-2H} \le 2 \le \frac{1}{2H-1}$ for $\frac{1}{2} < H \le \frac{3}{4}$. Summing up Lemmas \ref{Lambdalimit}, \ref{lemmaf1}, \ref{lemmaf2} and \ref{lemmaf3}, we then conclude that

\begin{equation}
\lim_{(\epsilon,\delta)\downarrow 0}\mathbb{E}\int_{\Delta_T}|\mathcal{W}^-(\epsilon,\delta;s,t) - \mathcal{W}(s,t)|^qdsdt=0,
\end{equation}
for $1\le q < \min \Big\{ \frac{1}{2H-1}; \frac{1}{2-2H}\Big\}$. By representing $\Lambda^-(\epsilon,\delta; s,t)$ as in (\ref{side2}) on the other simplex $\Delta_T^c=\{(s,t)\in [0,T]^2_\star;0 < t < s \le T\}$ and using (\ref{INTREP2}), we conclude the proof of Proposition \ref{Wpropo}.

\

\textbf{Proof of Theorem \ref{mainTH1}:} Apply Propositions \ref{deterProp} and \ref{Wpropo}.

\

\textbf{Proof of Theorem \ref{mainTH2}:} Let us assume that $Y^{\otimes 2} \in L^p(\Omega\times [0,T]^2; \mathbb{R}^{d\times d})$ for $\max \Big\{ \frac{1}{2H-1}; \frac{1}{2-2H} \Big\} < p \le \infty$. Then, for the conjugate exponent $q$ of $p$, we have $1\le q < \min \Big\{ \frac{1}{2H-1}; \frac{1}{2-2H}\Big\}$.
We recall that $\Lambda$ was defined in  \eqref{limIN}.
By using (\ref{fplit}) and applying Cauchy-Schwarz, H\"older's inequality and Theorem \ref{mainTH1}, we have

$$\Bigg|\Big \langle I^-(\epsilon,Y,dB), I^-(\delta,Y,dB) \Big\rangle_{L^2(\mathbb{P})}  - \mathbb{E}\int_{[0,T]^2 }  \big\langle Y_{s}\otimes Y_{t} , \Lambda (s,t) \big\rangle dsdt\Bigg|
$$
$$ = \Bigg|\mathbb{E}\int_{[0,T]^2}  \big\langle Y_{s}\otimes Y_{t} , \Lambda^-(\epsilon,\delta;s,t) - \Lambda (s,t) \big\rangle dsdt\Bigg|$$
$$\le \mathbb{E}\int_{[0,T]^2 }  \Big|\big\langle Y_{s}\otimes Y_{t} , \Lambda^-(\epsilon,\delta;s,t) - \Lambda (s,t) \big\rangle \Big|dsdt $$
$$\le \mathbb{E}\int_{[0,T]^2 }  \big| Y_{s}\otimes Y_{t}\big| \big|\Lambda^-(\epsilon,\delta;s,t) - \Lambda (s,t) \big|dsdt $$
$$\le\Bigg(\int_{[0,T]^2}  \mathbb{E}\big| Y_{s}\otimes Y_{t}\big|^pdsdt\Bigg)^{\frac{1}{p}} \Bigg(\int_{[0,T]^2}\mathbb{E}\big|\Lambda^-(\epsilon,\delta;s,t) - \Lambda (s,t) \big|^qdsdt\Bigg)^{\frac{1}{q}} \rightarrow 0,$$
as $(\epsilon,\delta)\downarrow 0$. This concludes the proof.

\section{Proof of Corollary \ref{cormainTH2}}
This section is devoted to the proof of Corollary \ref{cormainTH2}. At first, we need the following $L^p$-estimate for $\Lambda$ defined in  \eqref{limIN}.  

\begin{proposition}\label{West}
Fix $q>1$, $v \in (0,T]$ and the FBM of dimension $d$. Then,  
$$\| \Lambda(s,t)\|_q\lesssim_{H,d,q} |t-s|^{2H-2} + (s\wedge t)^{H-1}|t-s|^{H-1},$$
for $(s,t) \in [0,v]^2$ with $s\neq t$ and $s\wedge t>0$.  
\end{proposition}
\begin{proof}
Fix $v \in (0,T]$, $q>1$. By definition,

$$|\Lambda(s,t)|^q\lesssim_{d,q,H} |t-s|^{q(2H-2)}+ |\mathcal{W}(s,t)|^q,$$
for $s\neq t$ and $s\wedge t>0$. By the equivalence of the $L^r(\mathbb{P})$-norms on the second-Wiener chaos, it is sufficient to estimate $\| \mathcal{W}(s,t)\|_2$. By definition (\ref{Wproc}) (see (\ref{INTREP1}) and (\ref{INTREP2})), we can write

\begin{eqnarray*}
\mathcal{W}^{ij}(s,t)&=& \eta_{11}(s,t)\eta_{21}(s,t)\big(B^{(i)}_sB^{(j)}_s-s^{2H}\delta_{ij}\big) \\
\nonumber&+& \eta_{11}(s,t)\eta_{22}(s,t) \big( B^{(i)}_s B^{(j)}_{s,t} - \varphi(s,t)\delta_{ij} \big)\\
\nonumber&+& \eta_{12}(s,t) \eta_{21}(s,t) \big( B^{(j)}_s B^{(i)}_{s,t} - \varphi(s,t)\delta_{ij} \big)\\
\nonumber&+& \eta_{12}(s,t)\eta_{22}(s,t)\big(B^{(i)}_{s,t} B^{(j)}_{s,t} - |t-s|^{2H}\delta_{ij} \big),
\end{eqnarray*}
for $1\le i,j\le d$ for $(s,t) \in \Delta_T$. Here, 

$$\eta_{11}(s,t) =   \Big\{ s^{2H-1}|t-s|^{2H} - \varphi(s,t) |t-s|^{2H-1}\Big\}\frac{H}{\Theta_{s,t}}$$
$$\eta_{21}(s,t) =  \Big\{ \Big[t^{2H-1} - |t-s|^{2H-1}\Big]|t-s|^{2H} - \varphi(s,t)|t-s|^{2H-1}\Big\}\frac{H}{\Theta_{s,t}}$$
$$\eta_{12}(s,t) =  \Big\{ s^{2H}|t-s|^{2H-1} - s^{2H-1}\varphi(s,t)  \Big\}\frac{H}{\Theta_{s,t}}$$
$$\eta_{22}(s,t) =  \Big\{ s^{2H} |t-s|^{2H-1} - \varphi(s,t)\Big[t^{2H-1} - |t-s|^{2H-1}\Big]   \Big\}\frac{H}{\Theta_{s,t}},$$
for $0 < s < t\le u\le T$. By symmetry, H\"older's inequality and using the Gaussian property of $B$, it is sufficient to estimate 
$$
|\eta_{11}(s,t)\eta_{21}(s,t)|s^{2H} + |\eta_{12}(s,t)\eta_{22}(s,t)||t-s|^{2H}$$ 
$$+ \big|\eta_{11}(s,t) \eta_{22}(s,t) + \eta_{12}(s,t) \eta_{21}(s,t) \big| s^{H}|t-s|^H$$
for $0 < s < t \le u \le T$. Mean value theorem yields

\begin{equation}\label{deri}
t^{2H-1} -|t-s|^{2H-1}\lesssim_H s |t-s|^{2H-2},
\end{equation}
for $0 < s < t \le u \le T$. By using Lemma \ref{fhlemma}, we observe 

\begin{eqnarray*}
|\eta_{12}(s,t)\eta_{22}(s,t)||t-s|^{2H}&\lesssim& |t-s|^{2H-2} + \varphi(s,t) |t^{2H-1} - |t-s|^{2H-1}|\times s^{-2H}|t-s|^{-1}\\
&+&  s^{-1}|t-s|^{-1}\varphi(s,t)\\
&+& s^{-2H-1}\varphi^2(s,t) |t^{2H-1} - |t-s|^{2H-1}| |t-s|^{-2H},
\end{eqnarray*}
$0 < s < t \le u \le T$. By using (\ref{Exf5}), we observe 

\begin{equation}\label{fn1}
\sup_{0 < s< t\le u} \frac{\varphi(s,t)}{s^{2H-1}|t-s|} < \infty.
\end{equation}
The estimates (\ref{deri}) and (\ref{fn1}) yield 
 
\begin{eqnarray*} 
\varphi(s,t) |t^{2H-1} - |t-s|^{2H-1}|\times s^{-2H}|t-s|^{-1}&\lesssim& |t-s|^{2H-2}\frac{\varphi(s,t)}{s^{2H-1}|t-s|}\\
&\lesssim_H& |t-s|^{2H-2}, 
\end{eqnarray*}
for $0 < s < t \le u \le T$. By using (\ref{Exf5}), we observe 

\begin{eqnarray}\label{fn2}
\frac{\varphi(s,t)}{s|t-s|}&\lesssim& t^{2H-2}\\
\nonumber&=& |t-s|^{2H-2} t^{2H-2} |t-s|^{2-2H}\\
\nonumber&\lesssim& |t-s|^{2H-2},
\end{eqnarray}
$0 < s < t \le u \le T$. By using (\ref{fn2}), we observe 

\begin{eqnarray*}
s^{-2H-1}\varphi^2(s,t) |t^{2H-1} - |t-s|^{2H-1}| |t-s|^{-2H}&\le& s^{-2}\varphi^2(s,t)|t-s|^{-2}|t-s|^{2-2H}\\
&\lesssim& t^{2(2H-2)}|t-s|^{2-2H}\\
&\le&t^{2H-2}\\
&\lesssim& |t-s|^{2H-2}, 
\end{eqnarray*}
for $0 < s < t \le u \le T$. Similarly, one can easily check 

$$|\eta_{11}(s,t)\eta_{21}(s,t)s^{2H}|\lesssim |t-s|^{2H-2},$$
for $0 < s < t \le u \le T$. By using Lemma \ref{fhlemma}, (\ref{deri}) and (\ref{Exf5}), we observe 

\begin{eqnarray*}
|\eta_{11}(s,t)\eta_{22}(s,t)s^{H}|t-s|^H|&\lesssim& s^{H-1}|t-s|^{H-1} + s^{-H} \varphi(s,t)|t-s|^{H-2}\\
&+&s^{-H-1}|t-s|^{-H}\varphi(s,t) | t^{2H-1} - (t-s)^{2H-1}|\\
&+& s^{-3H}\varphi^2(s,t)| t^{2H-1} - (t-s)^{2H-1}| |t-s|^{-H-1}\\
&\lesssim& s^{H-1}|t-s|^{H-1} + |t-s|^{2H-2},
\end{eqnarray*}  
for $0 < s < t \le u \le T$. Similarly, 

$$
|\eta_{12}(s,t)\eta_{21}(s,t)s^{H}|t-s|^H|\lesssim s^{H-1}|t-s|^{H-1} + |t-s|^{2H-2},
$$
for $0 < s < t \le u \le T$. Summing up all the above estimates, we get 
$$\| \mathcal{W}(s,t)\|_{q} \lesssim_{H,d,q}s^{H-1}|t-s|^{H-1} + |t-s|^{2H-2},$$
for $0 < s < t \le u \le T$. This concludes the proof. 
\end{proof}

Let us consider an integrand  $Y_\cdot = g(\cdot, B_\cdot) \in L^p(\Omega\times [0,T]^2; \mathbb{R}^{d\times d})$ with exponent $\max \{ (2H-1)^{-1}; (2-2H)^{-1}\} < p \le \infty$ and let $q$ be the conjugate exponent of $p$ satisfying $1\le q < \min \{ (2H-1)^{-1}; (2-2H)^{-1}\}$. Since our setup is translation invariant, we restrict ourselves to the case $u=0$ for the proof of Corollary \ref{cormainTH2} without any loss of generality. In this case, a direct application of Theorem \ref{mainTH2}, Cauchy-Schwarz and H\"older's inequalities yield

\begin{equation}\label{upperINT}
\Big\|\int_0^v Y_r d^-B_r\Big\|^2_2\le \| Y^{\otimes_2}\|_p\Bigg(\int_{[0,v]^2}\mathbb{E}|\Lambda(s,t)|^qdsdt\Bigg)^{\frac{1}{q}} 
\end{equation}
where
$$\Lambda(s,t) = H(2H-1)|t-s|^{2H-2} I_{d\times d} + \mathcal{W}(s,t);$$  
for $(s,t) \in [0,v]^2$ with $s\neq t$ and $s\wedge t>0$. Proposition \ref{West} yields 

\begin{eqnarray}
\nonumber\int_{[0,v]^2}\| \Lambda(s,t)\|^q_qdsdt&\lesssim_{H,q,d}& \int_{[0,v]^2}|t-s|^{q(2H-2)}dsdt + \int_{[0,v]^2}s^{q(H-1)}|t-s|^{q(H-1)}dsdt\\
\label{ffi}&\lesssim_{H,q,d}& v^{q(2H-2)+2} + v^{2q(H-1)+2}
\end{eqnarray}
and hence by applying (\ref{ffi}) into (\ref{upperINT}), we conclude the proof of Corollary \ref{cormainTH2}.  

\begin{remark}\label{FR}
From Proposition \ref{West} and the fact that $\frac{1}{1-H} > \frac{1}{2-2H}$ for $\frac{1}{2} < H < 1$, we observe $\Lambda \in L^q (\Omega\times [0,T]^2; \mathbb{R}^{d\times d})$ for $1\le q < \frac{1}{2-2H}$. However, we are not able to prove $L^q$-convergence of $\Lambda^{-}(\epsilon,\delta, \cdot)$ to $\Lambda$ for every $1\le q < \frac{1}{2-2H}$ in Theorem \ref{mainTH1}. As mentioned in the Introduction, Theorem \ref{mainTH2} allows us to construct a natural Hilbert space of processes $\mathcal{H}_R$ as the completion of processes of the form $Y_\cdot = g(\cdot, B_\cdot)$ w.r.t. semi-inner product 

$$(Y,Z)\mapsto \mathbb{E}\int_{[0,T]^2} \langle  Y_s\otimes Z_t, \Lambda(s,t)\rangle dsdt$$
and one can naturally extend the forward integral to $\mathcal{H}_R$. In this case, it is possible to prove that any $Y_\cdot = g(\cdot, B_\cdot) \in L^q (\Omega\times [0,T]^2; \mathbb{R}^{d\times d})$ for $1\le q < \frac{1}{2-2H}$ will be an element of $\mathcal{H}_R$. In this paper, we do not make this extension and we leave a detailed analysis to a future work.    
\end{remark}

\

\

\

\noindent \textbf{Acknowledgments:} The authors gratefully acknowledge funding by FAPDF which supports the project ``Rough paths, Malliavin calculus and related topics'' 00193-00000229 2021-21.
The research of FR was also partially supported by the  ANR-22-CE40-0015-01 project (SDAIM).

\bibliographystyle{acm}
\begin{quote}
\bibliography{../../BIBLIO_FILE_ORV/refORV}
\end{quote}
\end{document}